\documentclass[12pt]{amsart}
\usepackage[margin=1in]{geometry}
\usepackage{changepage, bm, amsmath, makecell, multirow, cite,
  setspace,amsfonts, graphicx, color}
\usepackage{amsthm}
\usepackage{verbatim}
\usepackage{algorithm}
\usepackage{algorithmic}
\numberwithin{equation}{section}
\usepackage[toc,page]{appendix}
\newtheorem{theorem}{Theorem}[section]
\newtheorem{remark}{Remark}[section]

\newtheorem{lemma}[theorem]{Lemma}
\newtheorem{definition}{Definition}[section]

\begin{document}
\title{Data-driven computation methods for quasi-stationary
  distribution and sensitivity analysis}
\author{Yao Li}
\address{Yao Li: Department of Mathematics and Statistics, University
  of Massachusetts Amherst, Amherst, MA, 01002, USA}
\email{yaoli@math.umass.edu}
\author{Yaping Yuan}
\address{Yaping Yuan: Department of Mathematics and Statistics, University
  of Massachusetts Amherst, Amherst, MA, 01002, USA}
\email{yuan@math.umass.edu}

\thanks{Authors are listed in alphabetical order. Yao Li is partially supported by NSF DMS-1813246.}


\begin{abstract}
  This paper studies computational methods for
  quasi-stationary distributions (QSDs). We first proposed a
  data-driven solver that solves Fokker-Planck equations for
  QSDs. Similar as the case of Fokker-Planck equations for invariant
  probability measures, we set up an optimization problem that
  minimizes the distance from a low-accuracy reference solution, under
  the constraint of satisfying the linear relation given by the
  discretized Fokker-Planck operator. Then we use coupling method to
  study the sensitivity of a QSD against either the change of boundary
  condition or the diffusion coefficient. The 1-Wasserstein distance
  between a QSD and the corresponding invariant probability measure
  can be quantitatively estimated. Some numerical results about both
  computation of QSDs and their sensitivity analysis are provided.
\end{abstract}
\maketitle

\section{Introduction}
Many models in various applications are described by Markov chains
with absorbing states. For example, any models with mass-action
kinetics, such as ecological models, epidemic models, and chemical
reaction models, are subject to the population-level randomness called
the demographic stochasticity, which leads to extinction in finite
time. There are also many dynamical systems that have interesting
short term dynamics but trivial long term dynamics, such as dynamical
systems with transient chaos \cite{lai2011transient}. A common way of capturing asymptotical properties of these
transient dynamics is the quasi-stationary distribution (QSD), which
is the conditional limiting distribution conditioning on not hitting the absorbing
set yet. However, most QSDs do not have a closed form. So numerical
solutions are necessary in various applications.

Computational methods for QSDs are not very well
developed. Although the relation between QSD and the Fokker-Planck
equation is well known, it is
not easy to use classical PDE solver to solve QSDs because of the
following two reasons. First a QSD is the eigenfunction of the Fokker-Planck
operator whose eigenvalue is unknown. The cost of solving
eigenfunction of a discretized Fokker-Planck operator is considerably high. Secondly
the boundary condition of the Fokker-Planck equation is unknown. We
usually have a mixture of unbounded domain and unknown boundary value
at the absorbing set. As a result, Monte Carlo simulations are more
commonly used. However the efficiency of the Monte Carlo simulation is
known to be low. To get the probability density function, one needs to
deal with undesired noise associated to the Monte Carlo
simulation. Methods like the kernel density estimator can smooth the
solution but also introduce undesired diffusions to the solution,
especially when a QSD is highly concentrated at the vicinity of some
low-dimensional sets.

The first goal of this paper is to extend the data-driven
Fokker-Planck solver developed by the first author to the case of
QSDs \cite{dobson2019efficient}. Similar to \cite{dobson2019efficient}, we need a reference
solution $\mathbf{v}$ generated by the Monte Carlo simulation. Then we
discretize the Fokker-Planck operator in a numerical domain $D$
without the boundary condition. Because of the lack of boundary
conditions, the discretization only gives an underdetermined linear
system, denoted by $\mathbf{Au} = 0$. Then we minimize $\| \mathbf{v}
- \mathbf{u} \|$ in the null space of $\mathbf{A}$. As shown in \cite{dobson2021using},
this optimization problem projects the error terms of $\mathbf{v}$ to
a low dimensional linear subspace, which significantly reduces its
norm. Our numerical simulations show that this data-driven Fokker-Planck
solver can tolerate very high level of spatially uncorrelated error,
so the accuracy of $\mathbf{v}$ does not have to be high. The main
difference between QSD solver and the Fokker-Planck solver is that we
need a killing rate to find the QSD, which is obtained by a Monte Carlo
simulation. We find that the QSD is not very sensitive against small
error in the estimation of the killing rate.

The second goal of this paper is to study the sensitivity of
QSDs. Some modifications of either the boundary condition or the model
parameters can prevent the Markov process from hitting the absorbing state in
finite time. So the modified process would admit an invariant
probability measure instead of a QSD. It is important to understand
the difference between the QSD of a Markov process and the invariant
probability measure of its modification. For example, many ecological
models do not consider demographic noise because the population size
is large and the QSD is much harder to study. But would the demographic
noise completely change the asymptotical dynamics? More generally, a QSD
captures the transient dynamics of a stochastic differential
equation. If we separate a domain from the global dynamics by imposing
reflecting boundary condition, how would the local dynamics be
different from the corresponding transient dynamics? All these
require some sensitivity analysis with quantitative bounds.

The way of sensitivity analysis is similar to \cite{dobson2021using}. We need both finite time error
and the rate of contraction of the transition kernel of the Markov
process. The finite time error is given by both the 
killing rate and the change of model parameters (if any). Both cases
can be estimated easily. The rate of contraction is estimated by the
data-driven method proposed in \cite{li2020numerical}. We design a suitable
coupling scheme for the modified Markov process that admits an
invariant probability measure. Because of the coupling inequality, the exponential tail of the coupling
time can be used to estimate the rate of contraction. The sensitivity
analysis is demonstrated by several numerical examples. We can find
that the invariant probability measure of the modified process is a better approximation of
the QSD if (i) the speed of contraction is faster and (ii) the killing
rate is lower.

The organization of this paper is as follows. A short preliminary
about QSD, Fokker-Planck equation, simulation method, and coupling
method is provided in Section 2. Section 3 introduces the data-driven
QSD solver. The sensitivity analysis of QSD is given in Section
4. Section 5 is about numerical examples.

\section{Preliminary}
In this section, we provide some preliminaries to this paper, which
are about the quasi-stationary distribution (QSD), the Fokker-Planck
equation, the coupling method, and numerical simulations of the QSD.

\subsection{Quasi-stationary distribution}
We first give definition of the QSD and the exponential killing
rate $\lambda$ of a Markov process with an absorbing state. Let
$X=(X_t: t\geq 0)$ be a continuous-time Markov process taking values
in a measurable space $(\mathcal{X}, \mathcal{B}( \mathcal{X}))$. Let
$P^{t}(x, \cdot)$ be the transition kernel of $X$ such that
$P^{t}(x, A) = \mathbb{P}[X_{t} \in A \,|\, X_{0} = x]$ for all $A \in
\mathcal{B}$. Now assume that there exists an absorbing set
$\partial\mathcal{X}\subset \mathcal{X}$. The complement 
$\mathcal{X}^{a}:=\mathcal{X}\backslash\partial\mathcal{X}$ is the set
of allowed states.

The process $X_t$ is killed when it hits the
absorbing set, implying that $X_t\in \partial\mathcal{X}$ for all
$t>\tau$, where $\tau = \inf\{t>0: X_t\in\partial\mathcal{X} \}$ is
the hitting time of set $\partial\mathcal{X}$. Throughout this paper, we assume that the
process is almost surely killed in finite time, i.e. $\mathbb{P}[\tau<\infty]=1$.

For the sake of simplicity let $\mathbb{P}_x$
(resp. $\mathbb{P}_{\mu}$) be the probability
conditioning on the initial condition $x\in\mathcal{X}$ (resp. the
initial distribution $\mu$). 
\begin{definition}
A probability measure $\mu$ on $\mathcal{X}^{a}$ is called a
quasi-stationary distribution(QSD) for the Markov process $X_t$ with
an absorbing set
$\partial\mathcal{X}$, if for every measurable set $C\subset
\mathcal{X}^{a}$ 

\begin{equation}
    \mathbb{P}_{\mu}[X_t\in C|\tau>t]=\mu(C), \ t\geq 0,
    \label{1}
\end{equation}
or equivalently, if there is a probability measure $\mu$ exists such that
\begin{equation}
    \lim_{t\rightarrow\infty}\mathbb{P}_{\mu}[X_t\in C|\tau>t]=\mu(C).
    \label{2}
\end{equation}
in which case we also say that $\mu$ is a quasi-limiting distribution.
\end{definition}
\begin{remark}
When $\mu$ satisfies (\ref{2}), it is called a quasi-limiting distribution, or a Yaglom limit \cite{darroch1965quasi} .
\end{remark}

In the analysis of QSD, we are particularly interested in a parameter
$\lambda$, called the killing rate of the Markov
process. If the distribution of the killing time
$\mathbb{P}_x(\tau>t)$ has an exponential tail, then $\lambda$ is the
rate of this exponential tail. The following theorem shows that the
killing time is exponentially distributed when the process starts from a QSD\cite{collet2012quasi}.  

\begin{theorem}
Let $\mu$ be a QSD and when starting from $\mu$, the killing time $T$ is exponentially distributed, that is, 
\begin{equation*}
\exists \ \lambda=\lambda(\mu) \ \text{such that}\  \mathbb{P}_{\mu}[\tau>t]=e^{-\lambda t}, \ \forall t\geq0,
\end{equation*}
where $\lambda$ is called the killing rate of $X$.
\end{theorem}

Throughout this paper, we assume that $X$ admits a QSD denoted by
$\mu$ with a strictly positive killing rate $\lambda$. 

\subsection{Fokker-Planck equation}
Consider a stochastic differential equation
\begin{equation}
    \text{d}X_t=f(X_t)\text{d}t + \sigma(X_t)\text{d}W_t ,
\label{3}
\end{equation}
where $X_t\in\mathbb{R}^d$ and $X_t$ is killed when it hits the absorbing set $\partial\mathcal{X}\subset\mathbb{R}^n$; 
$f: \mathbb{R}^d\rightarrow\mathbb{R}^d$ is a continuous vector field; $\sigma$ is an $d\times d$ matrix-valued function; and $\text{d}W_t$ is the white noise in $\mathbb{R}^d$. The following well known theorem shows the existence and the uniqueness of the solution of equation \eqref{3}.
\begin{theorem}
Assume that there are two positive constants $C_1$ and $C_2$ such that the two functions $f$ and $\sigma$ in (\ref{3}) satisfy\\
(1) (Lipschitz condition) for all $x, y\in\mathbb{R}^d$ and $t$ 
\begin{equation*}
    |f(x)-f(y)|^2+|\sigma(x)-\sigma(y)|^2\leq C_1|x-y|^2;
\end{equation*}
(2) (Linear growth condition) for all $x, y\in\mathbb{R}^d$ and $t$
\begin{equation*}
    |f(x)|^2 + |\sigma(x)|^2 \leq C_2(1+|x|^2).
\end{equation*}
Then there exists a unique solution $X(t)$ to equation (\ref{3}).
\end{theorem}

There are quite a few recent results about the existence and
convergence of QSD. Since the theme of this paper is numerical
computations, in this paper we directly assume that  $X_t$ admits a unique QSD $\mu$ on set
$\mathcal{X}^{a}$ that is also the quasi-limit distribution. The detailed conditions are referred
in \cite{oksendal2003stochastic,karatzas2014brownian,van1991quasi,ferrari1992existence,van1995geomatric}.

Let $u$ be the probability density function of $\mu$. We refer \cite{anderson2012continuous} for the fact that $u$ satisfies
\begin{equation}
    -\lambda u = \mathcal{L}u = -\sum^d_{i=1}(f_i u)_{x_i} + \frac{1}{2}\sum^d_{i,j=1} (D_{i j} u)_{x_i x_j},
\label{4}
\end{equation}
where $D=\sigma^{T}\sigma$, and $\lambda$ is the killing rate. 
\subsection{Simulation algorithm of QSD}

It remains to review the simulation algorithm for QSDs. In order to
compute the QSD, one needs to numerically simulate a long trajectory
of $X$. Once $X_{t}$ hits the absorbing state, a new initial value is
sampled from the empirical QSD. The re-sampling step can be done in
two different ways. We can either use many independent trajectories that form
an empirical distribution \cite{barton1993uniform} or re-sample from the history
of a long trajectory \cite{robert2013monte}. In this paper we use the latter
approach.

Let $\hat{X}^\delta=\{\hat{X}^\delta_n,\ n\in\mathbb{Z}_{+}\}$ be a long numerical trajectory of the time-$\delta$ sample chain of $X_t$, where $\hat{X}^\delta_n$ is an numerical approximation of $X_{n\delta}$, then the Euler-Maruyama numerical scheme is given by
\begin{equation}
\hat{X}^\delta_{n+1} =\hat{X}^\delta_{n}+f(\hat{X}^\delta_{n})\delta+\sigma(\hat{X}^\delta_{n})(W_{(n+1)\delta}-W_{n\delta}), 
\label{euler}
\end{equation}
where $\hat{X}^\delta_{n}=X_0$, $W_{(n+1)\delta}-W_{n\delta}\sim
\mathcal{N}(0,\delta \mathrm{Id}_{d}),\ n\in\mathbb{Z}_{+}$ is a
$d$-dimensional normal random variable. 

Another widely used numerical scheme is called the Milstein scheme, which reads
\begin{equation*}
   \hat{X}^\delta_t=\hat{X}^\delta_{n}+f(\hat{X}^\delta_{n})(t-n\delta)+\sigma(\hat{X}^\delta_{n})(W_{t}-W_{n\delta})+
   \sigma(\hat{X}^\delta_{n})I L,
\end{equation*}
where $I$ is a $d\times d$ matrix with its $(i, j)$-th component being the double It$\hat{o}$ integral
\begin{equation*}
    I_{i,j}=\int^t_{n\delta}\int^s_{n\delta}dW^i(s_1)dW^j(s_2),
\end{equation*}
and $L\in\mathbb{R}^d$ is a vector of left operators with $i$-th component
\begin{equation*}
    u\mathcal{L}_i=\sum^d_{i=1}\sigma_{i,j}(\hat{X}^\delta_{n})\frac{\partial
    u}{\partial x_i}.
\end{equation*}

\begin{remark}
Under suitable assumptions of Lipschitz continuity and linear growth
conditions for  $f$ and $\sigma$,
the Euler-Maruyama approximation provides a convergence rate of order
1/2, while the Milstein scheme is an order 1 strong
approximation\cite{kloeden1992stochastic}. 
\end{remark}

For simplicity, we introduce the algorithm for $n=2$, specifically, we
solve $u$ in equation (\ref{4}) numerically on a 2D domain $D=[a_{0},
b_{0}]\times[a_1, b_1]$. Firstly, we construct an $N\times M$ grid on
$D$ with grid size
$h=\frac{b_{0}-a_{0}}{N}=\frac{b_1-a_1}{M}$. Each small box in the mesh
is denoted by $O_{i,j}=[a_{0}+(i-1)h, a_{0}+ih]\times[a_{1}+(j-1)h,
a_{1}+jh]$. Let $\mathbf{u} = \{u_{i,j}\}^{i=N,j=M}_{i=1,j=1}$ be the
numerical solution on $D$ that we are interested in, then $\mathbf{u}$
can be considered as a vector in $\mathbb{R}^{N\times M}$. Each
element $u_{i,j}$ approximates the density function $u$ at
the center of each $O_{i,j}$, with coordinate $(ih+a_{0}-h/2, jh+a_{1}-h/2)$. Generally
speaking, we count the number of $\hat{X}^\delta$ falling into each box
and set the normalized value as the approximate probability density at $O_{i,j}$. As we
are interested in the QSD, the main difference from the general Monte
Carlo is that the Markov process will be regenerated as the way of its
empirical distribution uniformly once it is killed. The details of
simulation is shown in Algorithm 1 as following. 

\begin{algorithm}
  \caption{Monte Carlo for QSD}
  \label{Alg:QSD}

  \begin{algorithmic}[l]
    \REQUIRE Equation \eqref{euler}
and the grid 
\ENSURE A Monte Carlo approximation $\mathbf{u}=\{u_{i,j}\}$. Sample
size $N_{s}$.
\STATE Pick any initial value $X_0\notin\partial\mathcal{X}$ in $D$
\FOR{$ n = 1\  \text{to}\  N_{s}$}
\STATE Use $\hat{X}^{\delta}_{n}$ and equation \eqref{euler} to compute $\hat{X}^\delta_{n+1}$
\STATE Record the coordinates of the small box $O_{i,j}$ where $\hat{X}^\delta_{n+1}$ stands, say $i^*, j^*$
\IF{$\hat{X}^\delta_{n+1}\notin\partial\mathcal{X}$}
\STATE $u_{i^*, j^*}\leftarrow u_{i^*, j^*}+1$
\ELSE
\STATE  $\hat{X}^\delta_{n+1}=\hat{X}^\delta_{\lfloor U*n \rfloor}$, \ where $U$ is a uniformly distributed random variable
\ENDIF
\ENDFOR
\STATE Return $u_{i,j}/N_{s} h^2$ for all $i, j$ as the approximation solution.
\end{algorithmic}
\end{algorithm}

Sometimes the Euler-Maruyama method underestimates the probability
that $X$ moves to the absorbing set, especially when $X_t$ is close to
$\partial\mathcal{X}$. This problem can be fixed by introducing the
Brownian bridge correction. We refer to \cite{benaim2018stochastic} for 
details. For a sample falling into a small box which are closed to
$\partial\mathcal{X}$, the probability of them falling into the trap
$\partial\mathcal{X}$ is relatively high. In fact, this probability is
exponentially distributed and the rate is related to the distance from
$\partial\mathcal{X}$. Let $B^T_B=W_t-\frac{t}{T}W_T$ be the Brownian
Bridge on the interval $[0,T]$. In the 1D case, the law of the infimum
and the supremum of the Brownian Bridge can be computed as follows:
for every $z\geq\max(x,y)$ 
\begin{equation}
    \mathbb{P}[\sup_{t\in[0, T]}(x+(y-x)\frac{t}{T}+\phi B^T_B)\leq z] = 1-\exp(-\frac{2}{T\phi^2}(z-x)(z-y)),
    \label{7}
\end{equation}
where $x=\hat{X}^\delta_{n}\in\mathcal{X}^a,
y=\hat{X}^\delta_{n+1}\in\mathcal{X}^a$, and $\phi = \sigma(\hat{X}^{\delta}_{n})$ is the strength
coefficient of Brownian Bridge.  This means that at each
step $n$, if $\hat{X}^\delta_{n+1}\in\mathcal{X}^a$, one can compute,
with the help of the above properties, a Bernoulli random variable $G$
with the parameter 
\begin{equation*}
    p=\mathbb{P}[\exists t\in(n\delta, (n+1)\delta),
    \hat{X_t}\in\partial\mathcal{X}|x=\hat{X}^\delta_{n},
    y=\hat{X}^\delta_{n+1}]\ \ (\text{If}\  G=1, \text{the process is killed}).
\end{equation*}

\subsection{Coupling Method}
The coupling method is used for the sensitivity analysis of QSDs. 
\begin{definition}
\textbf{(Coupling of probability measures)}
Let $\mu$ and $\nu$ be two probability measures on a probability space $(\Omega, \mathcal{F})$. A probability measure $\gamma$ on $(\Omega\times\Omega, \mathcal{F}\times\mathcal{F})$ is called a coupling of $\mu$ and $\nu$, if two marginals of $\gamma$ coincide with $\mu$ and $\nu$ respectively. 
\end{definition}
The definition of coupling can be extended to any two random variables
that take value in the same state space. Now consider two Markov processes
$X=(X_t:t\geq 0)$ and $Y=(Y_t: t\geq 0)$ with the same transition
kernel $P$. A coupling of $X$ and $Y$ is a stochastic process
$(\widetilde{X}, \widetilde{Y})$ on the product state space $\mathcal{X}\times
\mathcal{X}$ such that 

(i) The marginal processes $\widetilde{X}$ and $\widetilde{Y}$ are Markov processes with the transition kernel $P$;

(ii) If $\widetilde{X_s}=\widetilde{Y_s}$, we have $\widetilde{X_t}=\widetilde{Y_t}$ for all $t>s$.

The first meeting time of $X_t$ and $Y_t$ is denoted as $\tau^C:=\inf_{t\geq 0}\{X_t=Y_t\}$, which is called the coupling time. The coupling $(\widetilde{X}, \widetilde{Y})$ is said to be successful if the coupling time is almost surely finite, i.e. $\mathbb{P}[\tau^C<\infty]=1$. 

In order to give an estimate of the sensitivity of the QSD, we need
the following two metrics.
\begin{definition}(Wasserstein distance)
Let $d$ be a metric on the state space $V$. For probability measures $\mu$ and $\nu$ on $V$, the Wasserstein distance between $\mu$ and $\nu$ for $d$ is given by
\begin{equation}
\begin{split}
    d_w(\mu, \nu) &= \inf\{\mathbb{E}_{\gamma}[d(x, y)]:\ \gamma\  \text{is a coupling of } \mu\  and \ \nu.\}\\
    &=\inf\{\int d(x, y)\gamma(dx, dy): \gamma \  \text{is a coupling of } \mu\  and \ \nu. \}
\end{split}
\label{9}
\end{equation}
\end{definition}

In this paper, without further specification, we assume that the
1-Wasserstein distance is induced by $d(x,y)= \max\{1, \|x-y\|\}.$
\begin{definition}(Total variation distance)
Let $\mu$ and $\nu$ be probability measures on $(\Omega, \mathcal{F})$. The total variation distance of $\mu$ and $\nu$ is
\begin{equation*}
    \|\mu-\nu\|_{TV} := \sup_{C\in\mathcal{F}}|\mu(C)-\nu(C)|.
\end{equation*}
\end{definition}

\begin{lemma}(Coupling inequality)
For the coupling given above and the Wasserstein distance induced by the distance given in (\ref{9}), we have
\begin{equation*}
\mathbb{P}[\tau^C>t]=\mathbb{P}[X^t\neq Y^t]\geq d_w(P^t(x,\cdot), P^t(y,\cdot)).  \end{equation*}
\end{lemma}
\label{l1}
\begin{proof}
By the definition of Wasserstein distance,
\begin{equation*}
\begin{split}
    d_w(P^t(x,\cdot), P^t(y,\cdot))&\leq \int d(x,y)\mathbb{P}[(\widetilde{X^t},\widetilde{Y^t})\in (dx,dy)]\\
    &=\int_{x\neq y}d(x,y)\mathbb{P}[(\widetilde{X^t},\widetilde{Y^t})\in (dx,dy)]\\
    &\leq \int_{x\neq y}\mathbb{P}[](\widetilde{X^t},\widetilde{Y^t})\in (dx,dy)]\\
    &= \mathbb{P}[\widetilde{X^t}\neq\widetilde{Y^t}].
\end{split}
\end{equation*}
\end{proof}

Consider a Markov coupling $(\widetilde{X},
\widetilde{Y})$ where $\widetilde{X}_{t}$ and
$\widetilde{Y}_{t}$ are two numerical trajectories of the stochastic
differential equation described in \eqref{euler}. Theoretically, there are many ways to
make stochastic differential equations couple. But since numerical computation always has
errors, two numerical trajectories may miss each other when the true
trajectories couple. Hence we need to apply a mixture of the following
coupling methods in practice.

\textbf{Independent coupling.} Independent coupling means the noise
term in the two marginal processes $X_t$ and $Y_t$ are independent
when running the coupling process $(\widetilde{X},
\widetilde{Y})$. That is 
\begin{equation*}
\begin{split}
   X^\delta_{n+1} &= f(X^\delta_n) + (W^{(1)}_{(n+1)\delta}- W^{(1)}_{n\delta}) \\
   Y^\delta_{n+1} &= f(Y^\delta_n) + (W^{(2)}_{(n+1)\delta}- W^{(2)}_{n\delta}),
\end{split}
\end{equation*}
where $(W^{(1)}_{(n+1)\delta}- W^{(1)}_{n\delta})$ and $(W^{(2)}_{(n+1)\delta}- W^{(2)}_{n\delta})$ are independent random variables for each $n$. 

\textbf{Reflection coupling} Two Wiener processes meet less often than
the 1D case when the state space has higher dimensions. This fact
makes the independent coupling less effective. The reflection coupling
is introduced to avoid this case. Take the Euler-Maruyama scheme of
the SDE 
\begin{equation*}
    dX_t=f(X_t)dt + \sigma dW_t
\end{equation*}
as an example, where $\sigma$ is a constant matrix. The Euler-Maruyama scheme of $X_t$ reads as
\begin{equation*}
    \hat{X}^\delta_{n+1}=\hat{X}^\delta_n+f(\hat{X}^\delta_n)\delta+\sigma(W_{(n+1)\delta}-W_{n\delta}),
\end{equation*}
where $W$ is a standard Wiener process. The reflection coupling means that we run $\hat{X}^\delta_n$ as
\begin{equation*}
    \hat{X}^\delta_{n+1}=\hat{X}^\delta_n+f(\hat{X}^\delta_n)\delta+\sigma(W_{(n+1)\delta}-W_{n\delta}),
\end{equation*}
while run $\hat{Y}^\delta_n$ as
\begin{equation*}
  \hat{Y}^\delta_{n+1}=\hat{Y}^\delta_n+f(\hat{Y}^\delta_n)\delta+\sigma P (W_{(n+1)\delta}-W_{n\delta}),  
\end{equation*}
where $P=I-2e_ne^T_n$ is a projection matrix with
\begin{equation*}
    e_n=\frac{\sigma^{-1}(\hat{X}^\delta_n-\hat{Y}^\delta_n)}{\|\sigma^{-1}(\hat{X}^\delta_n-\hat{Y}^\delta_n)\|}.
\end{equation*}
Nontechnically, reflecting coupling means that the noise term is reflected against the hyperplane that
orthogonally passes the midpoint of the line segment connecting
$\hat{X}^\delta_n$ and $\hat{Y}^\delta_n$. In particular, $e_n=-1$
when the state space is 1D. 

\textbf{Maximal coupling} Above coupling schemes can bring $\hat{X}^\delta_n$
moves close to $\hat{Y}^\delta_n$ when running numerical simulations.
However, a mechanism is required to make
$\hat{X}^\delta_{n+1}=\hat{Y}^\delta_{n+1}$ with certain positive
probability. That's why the maximal coupling is involved. One
can couple two trajectories whenever the probability distributions of
their next step have enough overlap. Denote $p^{(x)}(x)$ and
$p^{(y)}(x)$ as the probability density functions of $\hat{X}_{n+1}$
and $\hat{Y}_{n+1}$ respectively. The implementation of the maximal coupling is
described in the following algorithm.

\begin{algorithm}
\caption{Maximal coupling}
\begin{algorithmic}[l]
  \REQUIRE $\hat{X}^\delta_n$ and $\hat{Y}^\delta_n$ \\
\ENSURE $\hat{X}^\delta_{n+1}$ and $\hat{Y}^\delta_{n+1}$, and $\tau^C$ if coupled
\STATE Compute probability density functions $p^{(x)}(z)$ and $p^{(y)}(z)$
\STATE Sample $\hat{X}^\delta_{n+1}$ and calculate $r=Up^{(x)}(\hat{X}^\delta_{n+1})$, where $U$ is uniformly distributed on [0,1]
\IF{$r<p^{(y)}(\hat{X}^\delta_{n+1})$}
\STATE $ \hat{Y}^\delta_{n+1}=\hat{X}^\delta_{n+1}, \tau^C=(n+1)\delta$
\ELSE
\STATE Sample $\hat{Y}^\delta_{n+1}$ and calculate $r'=V p^{(y)}(\hat{Y}^\delta_{n+1})$, where $V$ is uniformly distributed on [0,1]
\WHILE{$r'< p^{(x)}(\hat{Y}^\delta_{n+1})$}
\STATE Resample $\hat{Y}^\delta_{n+1}$ and $V$. Recalculate $r'=V p^{(y)}(\hat{Y}^\delta_{n+1})$
\ENDWHILE
\STATE $\tau^C$ is still undetermined
\ENDIF
\end{algorithmic}
\end{algorithm}

For discrete-time numerical schemes of SDEs, we use reflection coupling when
$\hat{X}^\delta_n$ and $\hat{Y}^\delta_n$ are far away from each
other, and maximal coupling when they are sufficiently close. The
threshold of changing coupling method is $2\sqrt{\delta}\|\sigma\|$ in
our simulation, that is, the maximal coupling is applied when the
distance between $\hat{X}^\delta_n$ and $\hat{Y}^\delta_n$ is smaller
than the threshold.

\section{Data-driven solver for QSD}
Recall that the probability density
function $u$ of QSD solves the Fokker-Planck equation $-\lambda u =
\mathcal{L}u$. The QSD solver consists of three components: an estimator
of the killing rate $\lambda$, a Monte Carlo simulator of QSD that produces a reference solution,
and an optimization problem similar as in \cite{li2018data}.

\subsection{Estimation of $\lambda$}
Let $\hat{X}^\delta=\{\hat{X}^\delta_n,\ n\in\mathbb{Z}_{+}\}$ be a
long numerical trajectory of  $X_t$ as described in Algorithm
\ref{Alg:QSD}. Let ${\bm \tau}=\{\tau_m\}^M_{m=0}$
be recordings of killing times of the numerical trajectory such that
$X_{t}$ hits $\partial\mathcal{X}$ at $\tau_{0}, \tau_{0}+\tau_{1},
\tau_{0}+\tau_{1}+\tau_{2}, \cdots$ when running Algorithm \ref{Alg:QSD}. Note that
${\bm \tau}$ is an 1D vector and each element in ${\bm \tau}$ is a
sample of the killing time. It is well known
that if the QSD $\mu$ exists for a Markov process, then there exists a constant
$\lambda > 0$ such that
$$
  \mathbb{P}_{\mu}[\tau > t] = e^{-\lambda t} \,.
$$
Therefore, we can simply assume that the killing times
${\bm \tau}$ be exponentially distributed and the rate can be approximated
by 
\begin{equation*}
\lambda=\frac{1}{\text{mean of}\  {\bm \tau}}. 
\end{equation*}

One pitfall of the previous approach is that Algorithm \ref{Alg:QSD}
only gives a QSD when the time approaches to infinity. It is
possible that ${\bm \tau}$ has not converged close enough to the
desired exponential distribution. So it remains
to check whether the limit is achieved. Our approach is to check the
exponential tail in a log-linear plot. After having ${\bm \tau}$, it is
easy to choose a sequence of times $t_{0}, t_{1}, \cdots, t_{n}$ and
calculate $n_{i} = | \{ \tau_{m} > t_{i} \,|\, 0 \leq m \leq M\} |$
for each $i = 0, \cdots n$. Then $p_{i} = n_{i}/M$ is an
estimator of $\mathbb{P}_{\mu}[ \tau > t_{i}]$. Now let $p_{i}^{u}$
(resp. $p_{i}^{l}$) be the upper (resp. lower) bound of the confidence
interval of $p_{i}$ such that
$$
  p_{i}^{u} = \tilde{p} + z \sqrt{\frac{ \tilde{p}}{\tilde{n}_{i}}(1 -
  \tilde{p})} \quad (\mbox{ resp. } p_{i}^{l} = \tilde{p} - z \sqrt{\frac{ \tilde{p}}{\tilde{n}_{i}}(1 -
  \tilde{p})} )\,,
$$
where $z = 1.96$, $\tilde{n}_{i} = n_{i} + z^{2}$ and $\tilde{p}=
\frac{1}{\tilde{n}}(n_i+\frac{z^2}{2})$ \cite{agresti1998approximate}. If $p_{i}^{l}
\leq e^{-\lambda t_{i}} \leq p_{i}^{u}$ for each $0 \leq i \leq n$, we
accept the estimate $\lambda$. Otherwise we need to run Algorithm
\ref{Alg:QSD} for longer time to eliminate the initial bias in ${\bm
  \tau}$.

\subsection{Data driven QSD solver.} The data driven solver for the
Fokker-Planck equation introduced in \cite{li2018data} can be modified
to solve the QSD for the stochastic differential equation
\eqref{3}. We use the same 2D setting in Section 2.3 to introduce the
algorithm. Let the domain $D$ and the boxes $\{ O_{i,j}\}_{i = 1, j
  =1}^{i = N, j = M}$ be the same as defined in Section 2.3. Let
$\textbf{u}$ be a vector in $\mathbb{R}^{N\times M}$ such that
$u_{ij}$ approximates the probability density function at the center
of the box $O_{i,j}$. As introduced in \cite{dobson2019efficient}, we consider $\textbf{u}$ as the
solution to the following linear system given by the spatial discretization of
the Fokker-Planck equation \eqref{4} with respect to each center point: 
\begin{equation}
    \mathbf{A_0u}=\lambda \mathbf{u},
    \label{5}
\end{equation}
where $\mathbf{A_0}$ is an $(N-2)(M-2)\times (NM)$ matrix, which is
called the discretized Fokker-Planck operator, and $\lambda$ is the
killing rate, which can be obtained by the way we mentioned in
previous subsection. More precisely, each row in $\mathbf{A_{0}}$
describes the finite difference scheme of equation \eqref{4} with
respect to a non-boundary point in the domain $D$. 

Similar to \cite{li2018data}, we need the Monte Carlo simulation
to produce a reference solution $\textbf{v}$, which can be obtained
via \textbf{Algorithm \ref{Alg:QSD}} in Section 2. Let
$\hat{X}^\delta=\{\hat{X}^\delta_{n}\}^N_{n=1}$ be a long numerical
trajectory of time -$\delta$ sample chain of process $X_t$ produced by
\textbf{Algorithm \ref{Alg:QSD}}, and let
$\textbf{v} = \{v_{i,j}\}^{i=N, j=M}_{i=1,j=1}$ such that 
\begin{equation*}
    v_{i,j}=\frac{1}{Nh^2}\sum^N_{n=1}\textbf{1}_{O_{i,j}}(\hat{X}^\delta_n)
\end{equation*}
It follows from the ergodicity of (\ref{3}) that $\textbf{v}$ is an approximate
solution of equation (\ref{4}) when the trajectory is sufficiently
long. However, as discussed in \cite{li2018data}, the trajectory
needs to be extremely long to make $\textbf{v}$ accurate
enough. Noting that the error term of $\textbf{v}$ has little spatial
correlation, we use the following optimization problem to improve the
accuracy of the solution. 
\begin{equation}
\begin{split}
\min\  &\Arrowvert \mathbf{u-v} \Arrowvert^2 \\
\text{subject to}\  & \mathbf{A_0u}=\lambda \mathbf{u}. 
\end{split}
\label{opt}
\end{equation}
The solution to the optimization problem \eqref{opt} is called the
least norm solution, which satisfies
$\mathbf{u}=\mathbf{v}-\mathbf{A^T(AA^T)^{-1}(Av)}$, within
$\mathbf{A=A_0}-\lambda \mathbf{I}$. \cite{li2018data}

An important method called the block data-driven solver is introduced
in \cite{dobson2019efficient}, in order to reduce the scale of numerical linear algebra
problem in high dimensional problems. By dividing domain $D$ into
$K\times L$ blocks $\{D_{k,l}\}^{k=1,l=1}_{k=K,l=L}$ and discretizing
the Fokker-Planck equation, the linear constraint on $D_{k,l}$ is 
\begin{equation*}
    \mathbf{A_{k,l}u^{k,l}} = \lambda \mathbf{u^{k,l}},
\end{equation*}
where $\mathbf{A_{k,l}}$ is an $(N/K-2)(M/L-2)\times(NM/KL)$ matrix. The optimization problem on $D_{k,l}$ is
\begin{equation*}
   \mathbf{u_{k,l}=-A^T_{k,l}(A_{k,l}A^T_{k,l})^{-1}A_{k,l}v_{k,l}+v_{k,l}},
\end{equation*}
where $\mathbf{v^{k,l}}$ is a reference solution obtained from the
Monte-Carlo simulation. Then the numerical solution to Fokker-Planck
equation (\ref{5}) is collage of all $\{u_{k,l}\}^{k=1,l=1}_{k=K,l=L}$
on all blocks. However, the optimization problem "pushes" most error
terms to the boundary of domain, which makes the solution is less
accurate near the boundary of each
block. Paper\\cite{dobson2019efficient} introduced the overlapping
block method and the shifting blocks method to reduce the interface error. The overlapping block method enlarges the
blocks and set the interior solution restricted to the original block
as new solution, while the shifting block method moves the interface
to the interior by shifting all blocks and recalculate the solution.

Note that in Section 3.1, we assume that $\lambda$ is a pre-determined
value given by the Monte Carlo simulation. Theoretically one can also
search for the minimum of $\| \mathbf{u - v} \|^{2}$ with respect to
both $\lambda$ and $\textbf{v}$. But empirically we find that the
result is not as accurate as using the killing rate $\lambda$ from the Monte Carlo
simulation, possibly because $\mathbf{v}$ has too much error.

One natural question is that how the simulation error in $\lambda$
would affect the solution $\mathbf{u}$ to the optimization problem
\eqref{opt}. Some linear algebraic calculation shows that the
optimization problem \eqref{opt}  is fairly robust against small
change of $\lambda$.

\begin{theorem}
  \label{lambda}
Let $\mathbf{u}$ and $\mathbf{u}_{1}$ be the solution to the
optimization problem \eqref{opt} with respect to killing rates
$\lambda$ and $\lambda_{1}$ respectively, where $|\lambda - \lambda_{1}|
= \epsilon \ll 1$. Then
$$
  \| \mathbf{u} - \mathbf{u}_{1} \|\leq 2 s_{min}^{-1}\epsilon
  \|\mathbf{v}\|  + O(\epsilon^{2}),
  $$
  where $s_{min}$ is the smallest singular value of $\mathbf{A}$. 
\end{theorem}

\begin{proof}
  Let $\mathbf{E}=\left[\epsilon
    \mathbf{I}_{(N-2)(M-2)}|\mathbf{0}\right]$ be an $(N-2)(M-2)\times
  (NM)$ perturbation matrix. Let $\mathbf{\widetilde{A}=A+E}$,
  $\mathbf{B=A^T(AA^T)^{-1}A}$ and
  $\mathbf{\widetilde{B}=\widetilde{A}^T(\widetilde{A}\widetilde{A}^T)^{-1}\widetilde{A}}$. Since
  $\mathbf{u} = \mathbf{v} - \mathbf{B} \mathbf{v}$ and
  $\mathbf{u}_{1} = \mathbf{v} - \mathbf{\widetilde{B}} \mathbf{v}$,
  it is sufficient to prove
\begin{equation*}
\mathbf{\Arrowvert B-\widetilde{B}} \Arrowvert\leq 2 s^{-1}_{\min}\epsilon + O(\epsilon^2). 
\end{equation*}

Note that 
\begin{equation*}
    \mathbf{\widetilde{A}\widetilde{A}^T}=\mathbf{(A+E)(A+E)^T=AA^T+A E^T+EA^T+E E^T}, 
\end{equation*}
and since $\|\mathbf{E E^T}\|$ is $O(\epsilon^2)$, we can neglect it when we consider the inverse matrix of $\mathbf{\widetilde{A}\widetilde{A}^T}$, then
\begin{equation*}
    \mathbf{(\widetilde{A}\widetilde{A}^T)^{-1}}=\mathbf{(AA^T)^{-1}-(AA^T)^{-1}(A E^T + E A^T)(AA^T)^{-1}},\\
\end{equation*}
and without considering the high order term $O(\epsilon^2)$, we can see 
\begin{equation*}
\begin{split}
   \mathbf{\widetilde{A}^T(\widetilde{A}\widetilde{A}^T)^{-1}\widetilde{A}} &= \mathbf{(A^T +E^T)((AA^T)^{-1}-(AA^T)^{-1}(A E^T+E A^T)(AA^T)^{-1})(A+E)}  \\
    &= \mathbf{A^T(AA^T)^{-1}A + (E^T(AA^T)^{-1}A-A^T(AA^T)^{-1}(A E^T)(AA^T)^{-1}A)} \\
    &\ + \mathbf{(A^T(AA^T)^{-1}E-A^T(AA^T)^{-1}(E A^T)(AA^T)^{-1}A)}  \\
    &= \mathbf{B+[E^T-A^T(AA^T)^{-1}(A E^T)](AA^T)^{-1}A}\\
    &\ +\mathbf{A^T(AA^T)^{-1}[E-(E A^T)(AA^T)^{-1}A]}
\end{split}
\end{equation*}
Consider the singular value decomposition(SVD) of matrix $\mathbf{A}$, i.e. $\mathbf{A=uS v^T}$, wherein
$\mathbf{S}=\begin{bmatrix}
\begin{array}{ccc|c}
s_1 &  &   &0\\
    & \ddots &   & \vdots\\
    &         &s_{(N-2)(M-2)} & 0\\
\end{array}
\end{bmatrix}$ is an $(N-2)(M-2)\times (NM)$ matrix  and both $\mathbf{u}\in\mathbb{R}^{(N-2)(M-2)\times (N-2)(M-2)}, \mathbf{v}\in\mathbb{R}^{(NM)\times (NM)}$ are orthogonal. Then $\mathbf{A^T(AA^T)^{-1}A=v D_1v^T}$, where $\mathbf{D_1}=\left[
	\begin{array}{cc}
	\mathbf{I}_{(N-2)(M-2)}& 0 \\ 
	0& 0
	\end{array}
	\right]_{(NM)\times (NM)}$,
and
\begin{equation*}
\begin{split}
   \mathbf{(E^T-A^T(AA^T)^{-1}AE^T)} &=\mathbf{(E^T-vD_1v^TE^T)}\\
   &=\mathbf{v D_2v^TE^T},\ \text{where}\ \mathbf{D_2=I-D_1}. \\
   \mathbf{E-(EA^T)(AA^T)^{-1}A} &=\mathbf{E-EA^T(AA^T)^{-1}A}\\
    &=\mathbf{E-E v D_1v^T}\\
    &=\mathbf{E v D_2v^T}.
\end{split}
\end{equation*}
Note that $\|\mathbf{vD_2v^TE^T}\|\leq\epsilon, \|\mathbf{E v D_2v^T}\|\leq\epsilon$.
Since $\mathbf{A^T(AA^T)^{-1}}$ and $\mathbf{(AA^T)^{-1}A}$ are two generalized inverse of $\mathbf{A}$, 
\begin{equation*}
    \mathbf{A^T(AA^T)^{-1}=v S^{*T}u^T, \ (AA^T)^{-1}A=u S^*v},
\end{equation*}
where 
\begin{equation}
\mathbf{S}^*
=\begin{bmatrix}
\begin{array}{ccc|c}
\frac{1}{s_1} &  &   &0\\
    & \ddots &   & \vdots\\
    &         &\frac{1}{s_{(N-2)(M-2)}} & 0\\
\end{array}
\end{bmatrix}
\end{equation}
and $\| \mathbf{A^T(AA^T)}^{-1}\|=\|\mathbf{(AA^T)^{-1}A}\|=\frac{1}{s_{\min}}$, hence we conclude that 
\begin{equation*}
\|\mathbf{B-\widetilde{B}}\|\leq 2{s_{\min}}^{-1}\epsilon+ O(\epsilon^2). 
\end{equation*}
\end{proof}

\begin{remark}
It is very difficult to estimate the minimum singular value of matrix
$\mathbf{A}$ analytically, even for the simplest case when the
Fokker-Planck equation is just a heat equation. But empirically we
find that $s_{min}^{-1}$ is usually not very large. For example,
$s_{min}^{-1}$ for the gradient flow with a double well potential is
Section 5.4 is $0.4988$, and $s_{min}^{-1}$ for the ``ring example'' in
Section 5.3 is only $0.2225$.
\end{remark}

\section{Sensitivity analysis of QSD} A stochastic differential equation has a
QSD usually because it has a natural absorbing state. For example, in ecological models,
this absorbing state is the natural boundary of the domain at which
the population of some species is zero. Obviously invariant
probability measures are easier to study than QSDs. One interesting question is
that if we slightly modify the equation such that it does not have
absorbing states any more, how can we quantify the difference between
QSD and the invariant probability measure after the modification? This
is called the sensitivity analysis of QSDs.

In this section, we focus on the difference between the QSD of a
stochastic differential equation $X=\{X_t,
t\in\mathbb{R}\}$ and the invariant probability measure of a
modification of $X$, denoted by $Y = \{ Y_{t}, t \in \mathbb{R}
\}$. For the sake of simplicity, this paper only compares the 
QSD (resp. invariant probability measure) of the numerical trajectory
of $X$ (resp. $Y$), denoted by $\hat{X} = \{
\hat{X}^{\delta}_{n}, n \in \mathbb{Z}_{+}\}$ (resp. $\hat{Y} = \{
\hat{Y}^{\delta}_{n}, n \in \mathbb{Z}_{+}\}$). Denote the QSD
(resp. invariant probability measure) of $\hat{X}$ (resp. $\hat{Y}$)
by $\hat{\mu}_{X}$ (resp. $\hat{\mu}_{Y}$) and the QSD
(resp. invariant probability measure) of the original SDE $X$ (resp. $Y$)
 by $\mu_{X}$ (resp. $\mu_{Y}$). The sensitivity of invariant probability
measure against time discretization has been addressed in \cite{dobson2021using}. When the time step size of
the time discretization is small enough, the invariant probability
measure $\mu_{X}$ is close to the numerical invariant probability
measure $\mu_{Y}$. The case of QSD is analogous. Hence
$d(\hat{\mu}_{X}, \hat{\mu}_{Y})$ is usually a good approximation of
$d(\mu_{X} , \mu_{Y})$.

We are mainly interested in the following two different modifications
of $X$.

\textbf{Case 1}: \textbf{Reflection at $\partial\mathcal{X}$} One easy
way to modify $X$ to eliminate the absorbing state is to add a
reflecting boundary. This method preserves the local dynamics but
changes the boundary condition. More precisely, the trajectory of
$\hat{X}$ follows that of $\hat{Y}$ until it hits the boundary
$\partial\mathcal{X}$. Without loss of generality assume $\partial\mathcal{X}$ is a
smooth manifold embedded in $\mathbb{R}^{d}$. If $\hat{Y}^{\delta}_{n}
= \hat{X}^{\delta}_{n}
\notin \partial\mathcal{X}$ but $\hat{X}^{\delta}_{n+1} \in \partial\mathcal{X}$, then
$\hat{Y}^{\delta}_{n+1}$ is the mirror reflection of
$\hat{X}^{\delta}_{n+1}$ against the boundary of
$\partial\mathcal{X}$. Still take the 2D case as an example. Let 
$\hat{Y}^{\delta}_n=\begin{bmatrix}
  y_{1,n} \\
  y_{2,n} 
\end{bmatrix}$. Without loss of generality, suppose $y_{1, n+1}$ first hits
the absorbing set $\partial\mathcal{X} = \{ x
\leq 0\}$, then 
\begin{equation*}
  \hat{Y}^{\delta}_{n+1}=\begin{bmatrix}
  -y_{1, n+1} \\
  y_{2, n+1} 
\end{bmatrix}  
\end{equation*}\\

\textbf{Case 2}: \textbf{Demographic noise in ecological models.} We
would also like to address the case of demographic noise in a class of
ecological models, such as population model and epidemic model. For
the sake of simplicity consider an 1D birth-death process with some
environment noise
\begin{equation}
  \label{BD}
  \mathrm{d}Y_{t} = f(Y_{t}) \mathrm{d}t + \sigma Y_{t}
  \mathrm{d}W_{t} \,,
  \end{equation}
  where $Y_{t} \in \mathbb{R}$. It is known that for suitable $f(Y_{t})$, the
  probability of $Y_{t}$ hitting zero in finite time is zero
  \cite{dobson2019efficient}. However, if we take the demographic noise, i.e., the
  randomness of birth/death events, into
  consideration, the birth-death process becomes 
  \begin{equation}
  \label{BD2}
  \mathrm{d}X_{t} = f(X_{t}) \mathrm{d}t + \sigma' X_{t}
  \mathrm{d}W^{(1)}_{t} + \epsilon \sqrt{X_{t}} \mathrm{d}W^{(2)}_{t} \,,
  \end{equation}
  where $\epsilon \ll 1$ is a small parameter that is proportional to
  $-1/2$-th power of the population size, $\sigma'$ is the new parameter
  that address the separation of environment noise and demographic
  noise. For example, if the steady state of $X_{t}$ is around $1$, we
  can choose $\sigma' = \sqrt{\sigma^{2} - \epsilon^{2}}$. Different from equation
  \eqref{BD}, equation \eqref{BD2} usually can hit the boundary $x = 0$ in
  finite time with probability one. Therefore, equation \eqref{BD}
  admits an invariant probability measure while equation \eqref{BD2}
  has a QSD. One very interesting question is that, if $\epsilon$ is
  sufficiently small, how different is the invariant probability measure of
  equation \eqref{BD} from the QSD of equation \eqref{BD2}? This is
  very important in the study of ecological models because
  theoretically every model is subject a small demographic noise. If
  the invariant probability measure is dramatically different from the
  QSD after adding a small demographic noise term, then the equation \eqref{BD}
  is not a good model due to high sensitivity, and we must study the
  equation \eqref{BD2} directly.

  \bigskip

  We roughly follow \cite{dobson2021using} to carry out the sensitivity
  analysis of QSD. Here we slightly modify $\hat{X}^{\delta}_{n}$ such
  that if $\hat{X}^{\delta}_{n} \in \partial \mathcal{X}$, we
  immediately re-sample $\hat{X}^{\delta}_{n}$ from the QSD
  $\hat{\mu}_{X}$. This new process, denoted by $\tilde{X}^{\delta}_{n}$,
  admits an invariant probability measure $\hat{\mu}_{X}$. Now denote
  the transition kernel of $\tilde{X}^{\delta}_{n}$ and
  $\hat{Y}^{\delta}_{n}$ by $P_{X}$ and $P_{Y}$ respectively.  Let
  $T>0$ be a fixed constant. Let $d_{w}(\cdot, \cdot)$ be the
  1-Wasserstein distance defined in Section 2. Motivated by \cite{johndrow2017error}, we can decompose
  $d_w(\mu_X, \mu_Y)$ via the following triangle inequality:
\begin{equation*}
    d_w(\mu_X,\mu_Y)\leq d_w(\mu_X P^T_X, \mu_X P^T_Y)+d_w(\mu_X P^T_Y, \mu_Y P^T_Y). 
\end{equation*}
If the transition kernel $P^T_Y$ has enough contraction such that
\begin{equation*}
    d_w(\mu_X P^T_Y, \mu_Y P^T_Y)\leq \alpha d_w(\mu_X,\mu_Y)
\end{equation*}
for some $\alpha<1$, then we have
\begin{equation*}
    d_w(\mu_X,\mu_Y)\leq\frac{d_w(\mu_X P^T_X, \mu_X P^T_Y)}{1-\alpha}.
  \end{equation*}

  Therefore, in order to estimate $d_w(\mu_X,\mu_Y)$, we need to look
  for suitable numerical estimators of the finite time error and the
  speed of contraction of $P^T_Y$. The finite time error can be easily
  estimated in both cases. And the speed of contraction $\alpha$ comes
  from the geometric ergodicity of the Markov process $\hat{Y}$. If
  our numerical estimation gives  
$$
  \mathrm{d}_{w}(\mu P^{T}_{Y}, \nu P^{T}_{Y}) \leq C e^{- \gamma T} \,,
$$
then we set $\alpha = e^{-\gamma T}$. As discussed in \cite{dobson2021using},
this is a quick way to estimate $\alpha$, and in practice it does not
differ from the ``true upper bound'' very much. The ``true upper
bound'' of $\alpha$ in \cite{dobson2021using} comes from the extreme value theory,
which is much more expensive to compute. 

\subsection{Estimation of contraction rate}
Similar as in \cite{li2020numerical}, we use the following coupling method to estimate the contraction
rate $\alpha$. Let $\hat{Z}^\delta_n=(\hat{Y}^1_n, \hat{Y}^2_n)$ be a
Markov process in $\mathbb{R}^{2d}$ such that $\hat{Y}^1_n$ and
$\hat{Y}^2_n$ are two copies of $\hat{Y}^{\delta}_{n}$. Let the first passage time to the ``diagonal''
hyperplane
$\{(\mathbf{x},\mathbf{y}\in\mathbb{R}^{2d})|\mathbf{y}=\mathbf{x}\}$
be the \textbf{coupling time}, which is denoted by $\tau^C$. Then by
Lemma \ref{l1} 
\begin{equation*}
    d_w(\mu_X P^t_Y, \mu_Y P^t_Y)\leq \mathbb{P}[\tau^C>t].
  \end{equation*}
As discussed in \cite{li2020numerical}, we need a hybrid coupling scheme to make
sure that two numerical trajectories can couple. Some coupling methods such as
reflection coupling or synchronous coupling are implemented in the
first phase to bring two numerical trajectories together. Then we
compare the probability density function for the next step and
couple these two numerical trajectories with the maximal possible
probability (called the maximal coupling). After doing this for many
times, we will have many samples of $\tau^{C}$ denote by ${\bm \tau^C}$. We use the exponential tail of $\mathbb{P}[ \tau^{C} > t]$ to estimate
the contraction rate $\alpha$. More precisely, we look for a constant
$\gamma > 0$ such that
$$
  -\gamma = \lim_{t \rightarrow \infty} \frac{1}{t}\log (\mathbb{P}[ \tau^{C} > t])
$$
if the limit exists. See \textbf{Algorithm \ref{Alg:coupling}} for the detail
of implementation of coupling. 
Note that we cannot simply compute the contraction rate start from $t=0$ because only the tail of coupling time can be considered as exponential distributed. In practice, we apply the same method as we compute the killing rate in section 3.1. After having ${\bm \tau^C}$, it is
easy to choose a sequence of times $t_{0}, t_{1}, \cdots, t_{n}$ and
calculate $n_{i} = | \{ \tau^C_{m} > t_{i} \,|\, 0 \leq m \leq M\} |$
for each $i = 0, \cdots n$. Then $p_{i} = n_{i}/M$ is an
estimator of $\mathbb{P}_{\mu}[ \tau^C > t_{i}]$. Now let $p_{i}^{u}$
(resp. $p_{i}^{l}$) be the upper (resp. lower) bound of the confidence
interval of $p_{i}$ such that
$$
  p_{i}^{u} = \tilde{p} + z \sqrt{\frac{ \tilde{p}}{\tilde{n}_{i}}(1 -
  \tilde{p})} \quad \mbox{ resp. } p_{i}^{l} = \tilde{p} - z \sqrt{\frac{ \tilde{p}}{\tilde{n}_{i}}(1 -
  \tilde{p})} \,,
$$
where $z = 1.96$, $\tilde{n}_{i} = n_{i} + z^{2}$ and $\tilde{p}=
\frac{1}{\tilde{n}}(n_i+\frac{z^2}{2})$. Let $t_{n}$ be the largest
time that we can still collect available samples. If there exist
constants $C$ and $i_{0} < n$ such that $p_{i}^{l}
\leq C e^{\gamma t_{i}} \leq p_{i}^{u}$ for each $i_{0} \leq i \leq
n$, we say that the exponential tail starts at $t = t_{i_{0}}$. We
accept the estimate of the exponential tail with rate $e^{-\gamma t}$
if the confidence interval $p_{i_{0}}^{u} - p_{i_{0}}^{l}$ is
sufficiently small, i.e., the estimate of coupling probability at $t =
t_{i_{0}}$ is sufficiently trustable. Otherwise we need to run Algorithm
\ref{Alg:QSD} for longer time to eliminate the initial bias in ${\bm
  \tau^C}$.

\begin{algorithm}
\caption{Estimation of contraction rate $\alpha$}
\begin{algorithmic}[l]
  \label{Alg:coupling}
  \REQUIRE Initial values $x,y\in K$ \\
\ENSURE An estimation of contraction rate $\alpha$ 
\STATE Choose threshold $d>0$
\FOR{$i=1\  \text{to}\  N_s$}
\STATE  $\tau^C_i=0, t=0, (\hat{Y}^1_t, \hat{Y}^2_t)=(x, y)$
\STATE Flag = 0
\WHILE{Flag=0}
\IF{$|\hat{Y}^1_t-\hat{Y}^2_t|>d$}
\STATE Compute $(\hat{Y}^1_{t+1}, \hat{Y}^2_{t+1})$ using reflection
coupling or independent coupling
\STATE $t\leftarrow t+1$
\ELSE
\STATE Compute $(\hat{Y}^1_{t+1}, \hat{Y}^2_{t+1})$ using maximal coupling
\IF{coupled successfully}
\STATE Flag=1
\STATE $\tau^C_i=t$
\ELSE
\STATE $t\leftarrow t+1$
\ENDIF
\ENDIF
\ENDWHILE
\ENDFOR
\STATE Use $\tau^C_1,\cdots,\tau^C_{N_s}$ to compute $\mathbb{P}(\tau^C>t)$
\STATE Fit the tail of $\log\mathbb{P}(\tau^C>t)$ versus $t$ by linear regression. Compute the slope $\gamma$.
\end{algorithmic}
\end{algorithm}

\subsection{Estimator of error terms}
It remains to estimate the finite time error $d_w(\mu_X P^T_X, \mu_X
P^T_Y)$. As we mentioned in the beginning of this section, we will
consider two different cases and estimate the finite time errors
respectively. 

\subsubsection{Case 1: Reflection at $\partial\mathcal{X}$}
Recall that the modified Markov process $\hat{Y}$ reflects at the
boundary $\partial\mathcal{X}$ when it hits the boundary. Hence two trajectories
\begin{equation*}
\begin{split}
    \tilde{X}^\delta_{n+1}&=\tilde{X}^\delta_{n}+f(\tilde{X}^\delta_{n})\delta+\sigma(\tilde{X}^\delta_{n})(W_{(n+1)\delta}-W_{n\delta})\\
    \hat{Y}^\delta_{n+1}&=\hat{Y}^\delta_{n}+f(\hat{Y}^\delta_{n})\delta+\sigma(\hat{Y}^\delta_{n})(W_{(n+1)\delta}-W_{n\delta})
\end{split}
\end{equation*}
are identical if we set the same noise in the simulation
process. $\tilde{X}$ only differs from $\hat{Y}$ when $\tilde{X}$ hits the
boundary $\partial \mathcal{X}$. When $\tilde{X}$ and $\hat{Y}$ hit the
boundary, $\tilde{X}$ is resampled from $\mu_{X}$, and $\hat{Y}$
reflects at the boundary. Hence the
finite time error $d_w(\mu_X P^T_X, \mu_X P^T_Y)$ is bounded from
above by the killing probability within the time interval $[0, T]$
when starting from $\mu_{X}$. 

In order to sample initial value $\mathbf{x}$ from the numerical
invariant measure $\mu_X$, we consider a long trajectory
$\{\tilde{X}^\delta_{n}\}$. The distance between $\tilde{X}$ and the
modified trajectory $\hat{Y}$ is recorded after time $T$, then we
let $\hat{Y} = \tilde{X}$ and restart. See the following algorithm for
the detail. 

\begin{algorithm}
  \caption{Estimate finite time error for Case 1}
  \label{Alg:W1case1}

  \begin{algorithmic}[l]
    \REQUIRE Initial value $\mathbf{x_1}$ \\
\ENSURE An estimator of $d_w(\mu_X P^T_X, \mu_X P^T_Y)$
\FOR{$i=1\  \text{to}\  N_s$}
\STATE  Using the same noise, simulate $\tilde{X}^\delta$ and $\hat{Y}^\delta$ with initial value $\mathbf{x_i}$ up to T
\STATE Set $K_i=0$, $d_i=0$
\IF{$\tau<T$}
\STATE Regenerate $\tilde{X}^\delta$ as its empirical distribution
\STATE $d_i=d(\tilde{X}^\delta_T, \hat{Y}^\delta_T)$
\ENDIF
\STATE Let $\mathbf{x_{i+1}}=\tilde{X}_T$ 
\ENDFOR
\STATE Return $\frac{1}{N_{s}}\sum^{N_{s}}_{i=1} d_i$
\end{algorithmic}
\end{algorithm}

When the number of samples is sufficiently large,
$\mathbf{x_1},\cdots, \mathbf{x_{N_{s}}}$ are from a long trajectory
of the time-$T$ skeleton of $\tilde{X}_T$. Hence they are approximately
sampled from $\mu_X$. The error term
$d_i=d(\tilde{X}^\delta_T, \hat{Y}^\delta_T)$ for
$\tilde{X}^\delta_0=\hat{Y}^\delta_0=\mathbf{x_i}$ estimates
$d(\tilde{X}_T,\hat{Y}_T)$. Let $\mu_{X}^{2}$ be the probability measure
on $\mathbb{R}^{d} \times \mathbb{R}^{d}$ that is supported by the
hyperplane $\{ (x, y) \in \mathbb{R}^{d} \times \mathbb{R}^{d} \,|\, x
= y\}$ such that $\mu_{X}^{2}( \{ (x, x) \,|\, x \in A\}) =
\mu_{X}(A)$ for any $A \in \mathcal{B}( \mathcal{X})$. Since the
pushforward map $\mu_{X}^{2}(P_{X}^{T} \times P_{Y}^{T})$ is a
coupling, it is easy to see that the output of \textbf{Algorithm
\ref{Alg:W1case1}} gives an upper bound of $d_{w}( \mu_{X} P_{X}^{T},
\mu_{X} P_{Y}^{T})$.

From the analysis above, we have the following lemma, which gives an
upper bound of the finite time error $d_w(\mu_X P^T_X, \mu_X
P^T_Y)$. 
\begin{lemma}
  \label{Lem:W1case1}
For the Wasserstein distance induced by the distance given in (\ref{9}), we have
\begin{equation*}
    d_w(\mu_X P^T_X, \mu_X P^T_Y)\leq\int_{\mathbb{R}^d}\mathbb{P}_{x}(\tau<T)\mu_X(dx),
\end{equation*}
where $x$ is the initial value with distribution $\mu_Y$.
\end{lemma}
\begin{proof}
Note that $\mu_{X}^{2}(P_{X}^{T} \times P_{Y}^{T})$ is a coupling of
$\mu_{X}P_{X}^{T}$ and $\mu_{X}P_{Y}^{T}$. From the definition of Wasserstein distance, we have
\begin{equation*}
\begin{split}
  d_w(\mu_X P^T_X, \mu_X P^T_Y) &\leq \int_{\mathbb{R}^{d} \times
    \mathbb{R}^{d}} d(x, y) \mu_{X}^{2}(P_{X}^{T} \times P_{Y}^{T})(
  \mathrm{d}x\mathrm{d}y)\\
&=  \int_{\mathbb{R}^d}d(xP^T_X, xP^T_Y)\mu_X(dx)\\
     &=\int_{\mathbb{R}^d}\mathbb{P}_x(\tau<T)d(\tilde{X}_T, \hat{Y}_T)\mu_X(dx)\\
     &\leq \int_{\mathbb{R}^d}\mathbb{P}_{x}(\tau^{K}<T)\mu_X(dx),
\end{split}
\end{equation*}
the inequality in the last step comes from the definition $d(x,y)=\max(1, \|x-y\|)$.
\end{proof}

\subsubsection{Case 2: Impact of a demographic noise
  $\epsilon\sqrt{X_t}\mathrm{d}W_t$} Another common way of modification in
ecological models is to add a demographic noise. Let $\hat{X}$ be the
numerical trajectory of the SDE with an additive demographic noise $\epsilon \sqrt{X_{t}}
\mathrm{d}W_{t}$. Let $\tilde{X}$ be the modification of $\hat{X}$
that resample from $\mu_{X}$ whenever hitting $\partial \mathcal{X}$
so that it admits $\mu_{X}$ as an invariant probability measure. Let
$\hat{Y}$ be the numerical trajectory of the SDE without demographic noise
so that $\hat{Y}$ admits an invariant probability measure. We have trajectories
\begin{equation*}
\begin{split}
    \tilde{X}^\delta_{n+1}&=\tilde{X}^\delta_{n}+f(\tilde{X}^\delta_{n})\delta+\sigma(\tilde{X}^\delta_{n})(W_{(n+1)\delta}-W_{n\delta})
    + \epsilon \sqrt{\hat{X}^{\delta}_{n}}(W'_{(n+1)\delta} - W'_{n\delta} )\\
    \hat{Y}^\delta_{n+1}&=\hat{Y}^\delta_{n}+f(\hat{Y}^\delta_{n})\delta+\sigma'(\hat{Y}^\delta_{n})(W_{(n+1)\delta}-W_{n\delta}) \,.
\end{split}
\end{equation*}
Here we assume that $\hat{Y}^\delta$ has zero probability to hit the
absorbing set $\partial\mathcal{X}$ in finite time. Different from the Case 1, we
will need to study the effect of the demographic noise. When
estimating the finite time error $d_w(\mu_X P^T_X, \mu_X P^T_Y)$, we 
still need to sample the initial value $\mathbf{x}$ from $\mu_X$ and
record the distance between these two trajectories $\tilde{X}^\delta$
and $\hat{Y}^\delta$ up to time $T$. The distance
between $\tilde{X}^\delta$ and $\hat{Y}^\delta$ can be decomposed into two
parts: one is from the killing and resampling, the other is from the
demographic noise. The first term is the same as in Case 1. The second term is due
to the nonzero demographic noise that can separate $\tilde{X}$ and
$\hat{Y}$ before the killing. In a population model, this effect is
more obvious when one species has small population, because $\sqrt{x}
\gg x$ when $0 < x \ll 1$. See the description of Algorithm
\ref{Alg:W1case2} for the full detail.

\begin{algorithm}
\caption{Estimate finite time error for Case 2}
\begin{algorithmic}[l]
  \label{Alg:W1case2}
  \REQUIRE Initial value $\mathbf{x_1}$ \\
\ENSURE An estimator of $d_w(\mu_X P^T_X, \mu_X P^T_Y)$
\FOR{$i=1\  \text{to}\  N_s$}
\STATE  Using the same noise, simulate $\tilde{X}^\delta$ and $\hat{Y}^\delta$ with initial value $\mathbf{x_i}$ up to T
\STATE Set $ \text{Flag} = 0$, $d_i=0$
\IF{$\tau < T$}
\STATE Regenerate $\tilde{X}^\delta$ as its empirical distribution
\STATE $\eta_i=d(\tilde{X}^\delta_T, \hat{Y}^\delta_T)$
\STATE Flag = 1
\ELSE
\STATE $\theta_i=d(\tilde{X}^\delta_T, \hat{Y}^\delta_T)$
\ENDIF
\STATE $d_i=\theta_i + \mathbf{1}_{\{\text{Flag}=1\}}(\eta_i-\theta_i)$
\STATE Let $\mathbf{x_{i+1}}=\tilde{X}_T$ 
\ENDFOR
\STATE Return $\frac{1}{N_{s}}\sum^{N_{s}}_{i=1} d_i$
\end{algorithmic}
\end{algorithm}

When $N_{s}$ is sufficiently large, $\mathbf{x_1},\cdots,
\mathbf{x_{N_{s}}}$ are from a long trajectory of the time-$T$
skeleton of $\tilde{X}_T$. Hence they are approximately sampled from
$\mu_X$. The error term $d_i$ for
$\tilde{X}^\delta_0=\hat{Y}^\delta_0=\mathbf{x_i}$ estimates
$d(\tilde{X}_T,\hat{Y}_T)$. A similar coupling argument shows that the
output of Algorithm \ref{Alg:W1case2} is an upper bound of $d_w(\mu_X
P^T_X, \mu_X P^T_Y)$.

For each initial value $x \in \mathbb{R}^{d}$, denote
$$
  \theta_{x} = \mathbb{E}_{x}[ d( \tilde{X}^{\delta}_{T},
  \hat{Y}^{\delta}_{T}) \,|\, \tilde{X}^{\delta}_{0} =
  \hat{Y}^{\delta}_{0} = x, \tau > T] \,.
$$
Similar as in Case 1, the
following lemma  gives an upper bound for the finite time error
$d_w(\mu_XP^T_X, \mu_X P^T_Y)$. 
\begin{lemma}
  \label{Lem:W1case2}
For the Wasserstein distance induced by the distance given in (\ref{9}), we have
\begin{equation*}
    d_w(\mu_X P^T_X, \mu_X P^T_Y)\leq\int\mathbb{P}_{x}(\tau<T)\mu_Y(dx) + \int \theta_x\mu_Y(dx),
\end{equation*}
where $x$ is the initial value with distribution $\mu_Y$ and $\theta_x$ .
\end{lemma}
\begin{proof}
Note that $\mu_{X}^{2}(P_{X}^{T} \times P_{Y}^{T})$ is a coupling of
$\mu_{X}P_{X}^{T}$ and $\mu_{X} P_{Y}^{T}$. From the definition of the
Wasserstein distance, we have
\begin{equation*}
    \begin{split}
      d_w(\mu_Y P^T_X, \mu_Y P^T_Y) &\leq \int_{\mathbb{R}^{d} \times
        \mathbb{R}^{d}} d(x,y) \mu_{X}^{2}(P_{X}^{T} \times
      P_{Y}^{T})( \mathrm{d}x \mathrm{d}y)\\
&=    \int_{\mathbb{R}^d}d(xP^T_X, xP^T_Y)\mu_X(dx)\\
     &=\int_{\mathbb{R}^d}\mathbb{P}_x(\tau<T)d(\tilde{X}_T,
     \hat{Y}_T)\mu_X(dx)+\int_{\mathbb{R}^d}\mathbb{P}_x(\tau>T)d(\tilde{X}_T,
     \hat{Y}_T)\mu_X(dx)\\ 
     &\leq \int_{\mathbb{R}^d}\mathbb{P}_{x}(\tau<T)\mu_X(dx) + \int \theta_x\mu_X(dx)
    \end{split}
\end{equation*}
according to the definition of $\theta_{x}$. 
\end{proof}

\section{Numerical Examples}
\subsection{Ornstein–Uhlenbeck process}
The first SDE example is the Ornstein–Uhlenbeck process :
\begin{equation}
 \text{d} X_t= \theta(\mu-X_t)\text{d}t + \sigma\text{d}W_t,   
\end{equation}
where $\theta>0$ and $\sigma>0$ are parameters, $\mu$ is a
constant. In addition, $W_t$ is a Wiener process, and $\sigma$ is the strength of the noise. In our simulation, we set $\theta=1,
\mu=2, \sigma=1$ and the absorbing set $\partial\mathcal{X}=(-\infty,
0] \cup [3, \infty)$. In
addition, we apply the Monte Carlo simulation with $512$  mesh points on
the interval $[0, 3]$. 

We first need to use Algorithm \ref{Alg:QSD} to estimate the survival
rate $\lambda$. Our simulation
uses Euler-Maruyama scheme with $\delta = 0.001$ and sample size $N =
10^{6}$ and $N = 10^{8}$ depending on the setting. All samples of killing times are recorded to
plot the exponential tail. The mean of killing times gives
an estimate $\lambda = -0.267176$. The exponential tail of $\mathbb{P}[ \tau
> t]$ vs. $t$, the upper and lower bound of the confidence interval,
and the plot of $e^{-\lambda t}$ are compared in Figure
\ref{F1}. We can see that the plot of $e^{- \lambda t}$ falls in
the confidence interval for all $t$. Hence the estimate of $\lambda$
is accepted. 

\begin{figure}[hbt!]
    \centering
    \includegraphics[width=\textwidth]{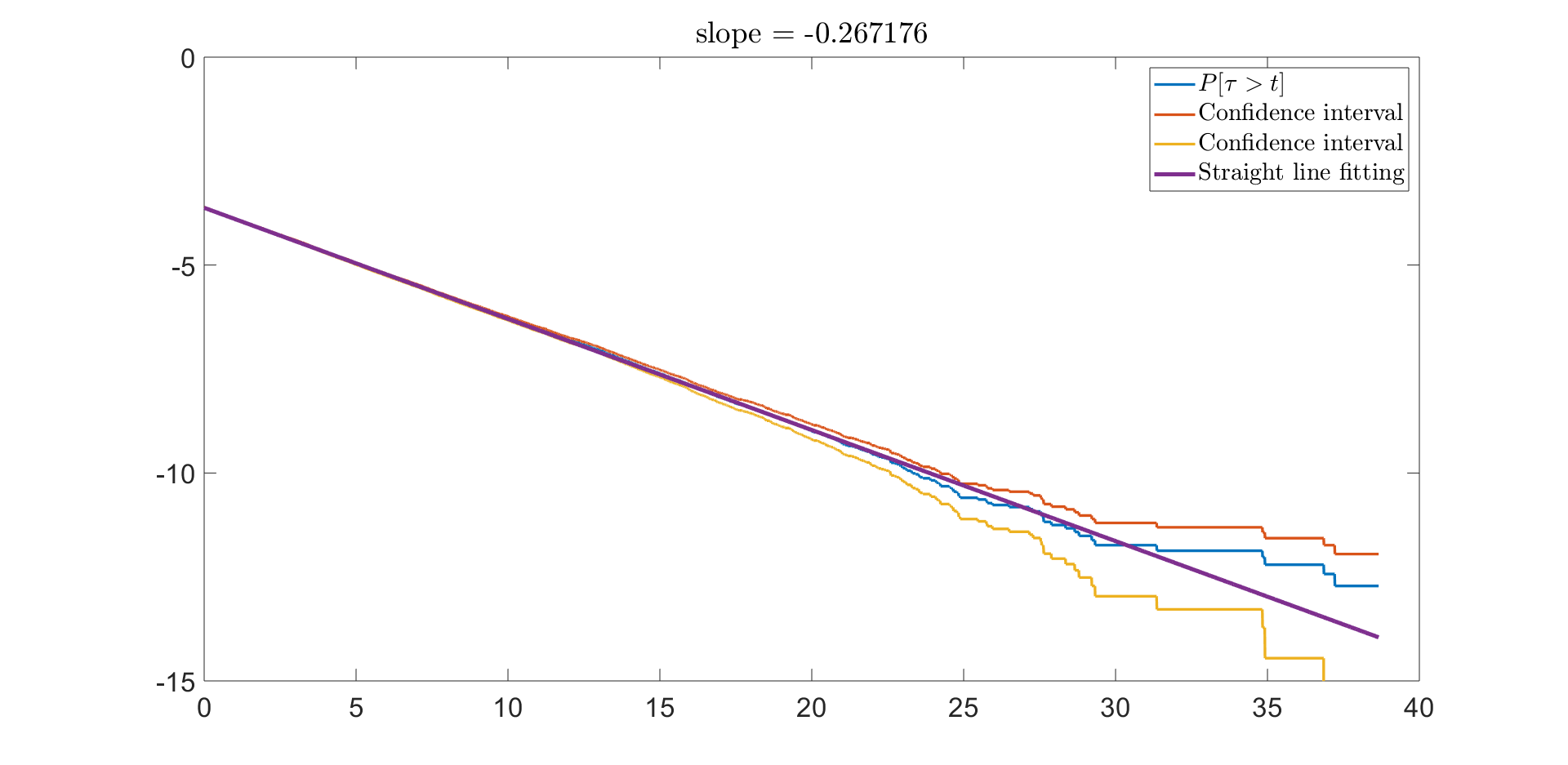}     \caption{Plot of $\mathbb{P}(\tau>t)$ vs. $t$, confidence interval(upper bound and lower bound) and function $y=e^{-\lambda t}$}.
    \label{F1}
  \end{figure}

Furthermore, we would like to show the robustness of our data-driven
QSD solver. The QSD is not explicit given so we use very large sample
size ($10^{10}$ samples) and very small step size ($10^{-4}$) to
obtain a much more accurate solution, which is served as the
benchmark. Then we compare the numerical solutions obtained by the Monte
Carlo method and the data-driven method for QSD with $N=10^6$ and
$N=10^8$ samples, respectively. The result is shown in the first
column of Figure \ref{F2}. The data-driven solver performs much better
than the Monte Carlo approximation for $N=10^6$ samples. It takes
$10^{8}$ samples for the direct Monte Carlo sampler to produce a
solution that looks as good as the QSD solver. Similar as the data-driven Fokker-Planck solver, our
data-driven QSD solver can tolerate high level error in Monte Carlo
simulation that has small spatial correlation. 

It remains to check the effect of Brownian Bridge. We apply
different time step sizes $\delta=0.01$ and $\delta=0.001$ for each
trajectory. We use $10^{7}$ samples for $\delta = 0.001$ and $10^{6}$
samples for $\delta = 0.01$ to make sure that the number of killing
events (for estimating the killing rate) are comparable. When $\delta
= 0.001$, the error is small with and without Brownian bridge
correction. But Brownian bridge correction obviously improves the
quality of solution when $\delta=0.01$. See the lower left panel of
Figure \ref{F2}. This is expected because,
with larger time step size, the probability that the Brownian bridge
hits the absorbing set $\partial \mathcal{X}$ gets higher.

\begin{figure}[hbt!]
    {\centering
    \includegraphics[width=1.1\textwidth]{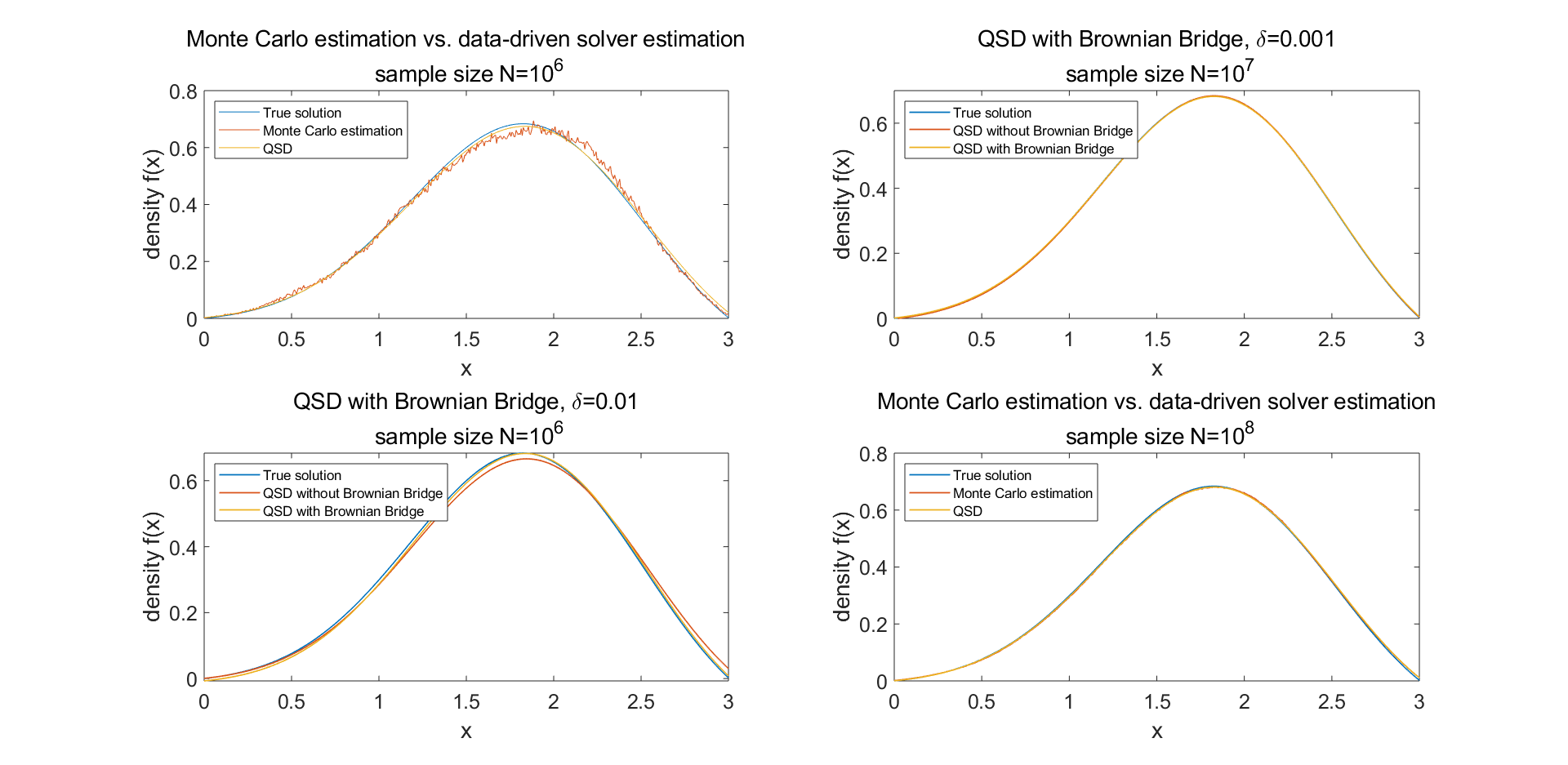}}  \caption{
    $\textbf{Left column}$: Monte Carlo estimation vs.
    data-driven solver estimation. $\textbf{Middle column}$: Effect of Brownian Bridge for sample size $N=10^6$ and $N=10^8$, with time step size 0.001. $\textbf{Right column}$:  Effect of Brownian Bridge for sample size $N=10^6$ and $N=10^8$, with time step size 0.01.}
    \label{F2}
  \end{figure}

\subsection{Wright-Fisher Diffusion}
The second numerical example is the Wright-Fisher diffusion model,
which describes the evolution of colony undergoing random mating,
possible under the additional actions of mutation and selection with
or without dominance \cite{huillet2007wright}. The Wright-Fisher model is an SDE
\begin{equation*}
    dX_t=-X_tdt + \sqrt{X_t(1-X_t)}dW_t,
\end{equation*}
where $W_t$ is a Wiener process and $\partial\mathcal{X}=\{0\}$ is the
absorbing set. By the analysis of \cite{huillet2007wright}, the Yaglom limit, i.e.,
the QSD, satisfies 
\begin{equation*}
    \lim_{t \rightarrow \infty}\mathbb{P}[X_t\in dy|\tau>t] =  2(1-y)dy.
\end{equation*}

The goal of this example is to show the effect of Brownian
bridge when the coefficient of noise is singular at the
boundary. Since the Euler-Maruyama scheme only has an order of
accuracy $0.5$, in the simulation, we apply the Milstein
scheme, which reads  
\begin{equation*}
  \hat{X}^\delta_{n+1}=\hat{X}^\delta_{n} - \hat{X}_{n}^{\delta}
  \delta + \sqrt{\hat{X}^{\delta}_{n}(1-
    \hat{X}^{\delta}_{n})}(W_{(n+1)\delta}-W_{n\delta})  +
  \frac{1}{4}(1 - 2
  \hat{X}^{\delta}_{n})[(W_{(n+1)\delta}-W_{n\delta})^{2} - \delta ]
\end{equation*}

One difficulty of using the Brownian bridge correction in
this model is that the coefficient of the Brownian motion is vanishing
at the boundary. Recall that the strength coefficient of Brownian
bridge is denoted by $\phi$. Since the coefficient of the Brownian
motion is vanishing, the original strength coefficient $\phi =
\sqrt{\hat{X}_{n}^{\delta}(1 - \hat{X}_{n}^{\delta})}$
can dramatically overestimate the probability of hitting the
boundary. In this example, it is not possible to explicitly calculate
the conditional distribution of the Brownian bridge that starts from
$x:=\hat{X}_{n}^{\delta}$ and ends at $y:=\hat{X}_{n+1}^{\delta}$. But we
find that empirically $\phi^{2} = \frac{1}{3}\min\{ x(1-x), y(1-y)\}$
can fix this problem. Since a similar vanishing coefficient of the
Brownian motion also appears in many ecological models, we will
implement this modified  Brownian bridge correction when simulating
these models. 

\begin{figure}[h]
    \centering
    \includegraphics[width=\textwidth]{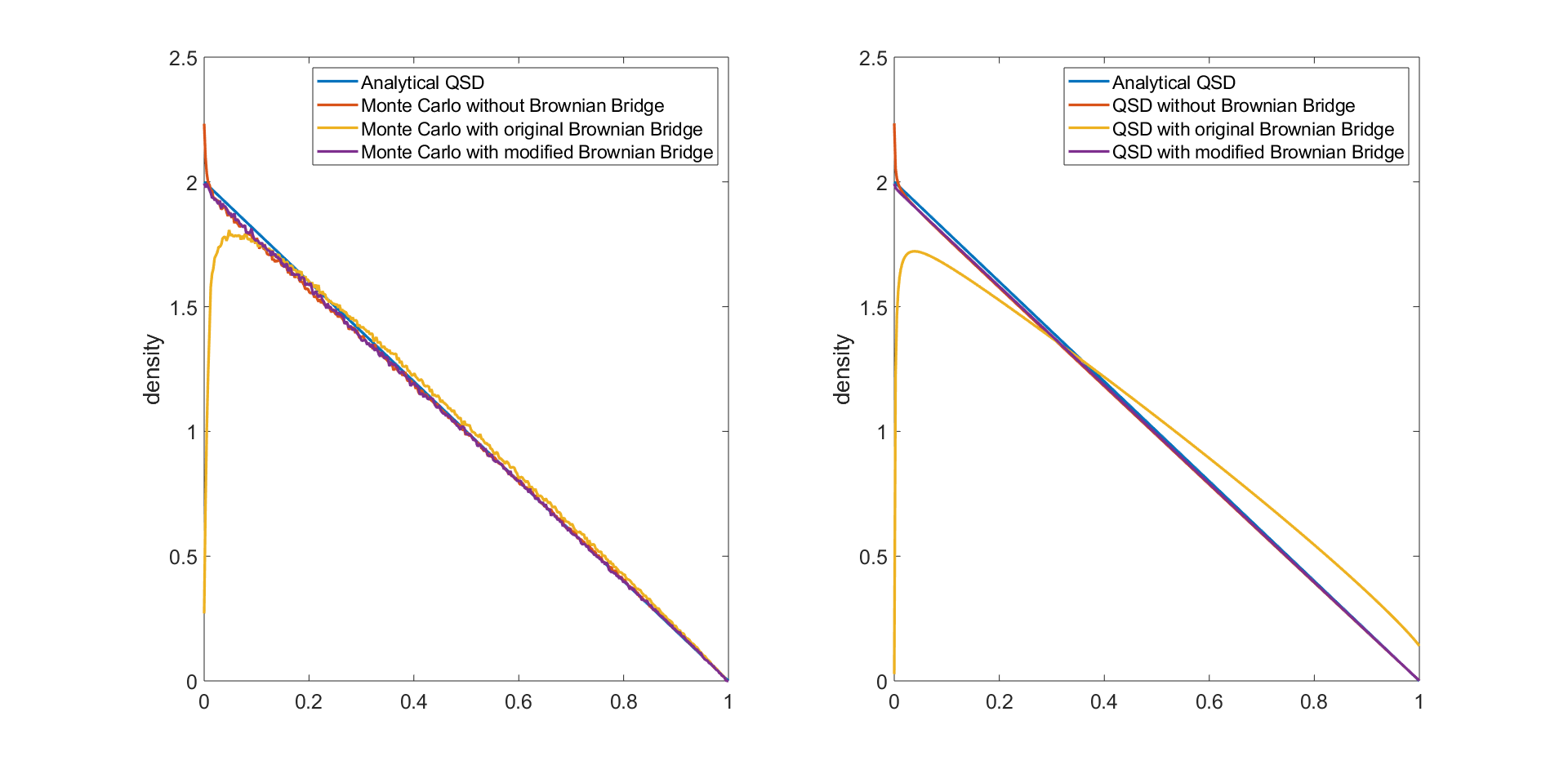}
     \caption{Effect of Brownian Bridge and a correction of Brownian
       Bridge.\ \textbf{Left}: Monte Carlo approximations without
       Brownian Bridge correction, with original Brownian Bridge
       correction, and with modified Brownian Bridge correction, in
       comparison to the analytical QSD. \textbf{Right}:
       Result from the data-driven QSD solver using the Monte Carlo
       simulation data from Left panel. }
     \label{F3}
\end{figure}

The effect of Brownian bridge is shown in the left side of Figure
\ref{F3}.  We compare the solutions obtained via Monte Carlo method and the
data-driven method with Brownian Bridge by running $10^7$ samples on
$[0,1]$ with time step size $\delta=0.01$. The Monte Carlo
approximation  is far from the true density function of
Beta(1,2) near $x = 0$, while the use of the original Brownian Bridge only makes
things worse. The modified Brownian Bridge solves this boundary effect problem
reasonably well. The output of the data-driven QSD solver has a
similar result (Figure \ref{F3} Right).  Let $x=\hat{X}_{n}^{\delta}$ and
$y=\hat{X}_{n+1}^{\delta}$. One can see that the numerical QSD is much
closer to the true distribution if we replace the strength 
coefficient of the Brownian bridge $\phi^{2} = x(1-x)$ by the
modified strength coefficient $\phi^{2} = \frac{1}{3}\min\{
x(1-x), y(1-y)\}$.

\subsection{Ring density function}
Consider the following stochastic differential equation:
\begin{equation*}
  \begin{split}
      dX&=(-4X(X^2+X^2-1)+Y)dt +\epsilon dW^x_t\\
      dY&=(-4X(X^2+Y^2-1)-X)dt +\epsilon dW^y_t,
  \end{split}  
\end{equation*}
where $W^x_t$ and $W^y_t$ are independent Wiener processes. In the
simulation, we set the strength of noise $\epsilon=1$. 

We first look at the approximation obtained by Monte Carlo method with
$256\times256$ mesh points on the domain
$D=[-1.5,1.5]\times[-1.5,1.5]$. The simulation uses step size $\delta=0.001$ and
$N=10^8$ samples. (See upper left panel in Figure
\ref{F4}). The Monte Carlo approximation has too much noise to be
useful. The quality of this solution can be significantly improved by
using the data-driven QSD solver. See upper right panel
in Figure \ref{F4}.
\begin{figure}[h]
    \centering
     \includegraphics[height = 18cm, width=\textwidth]{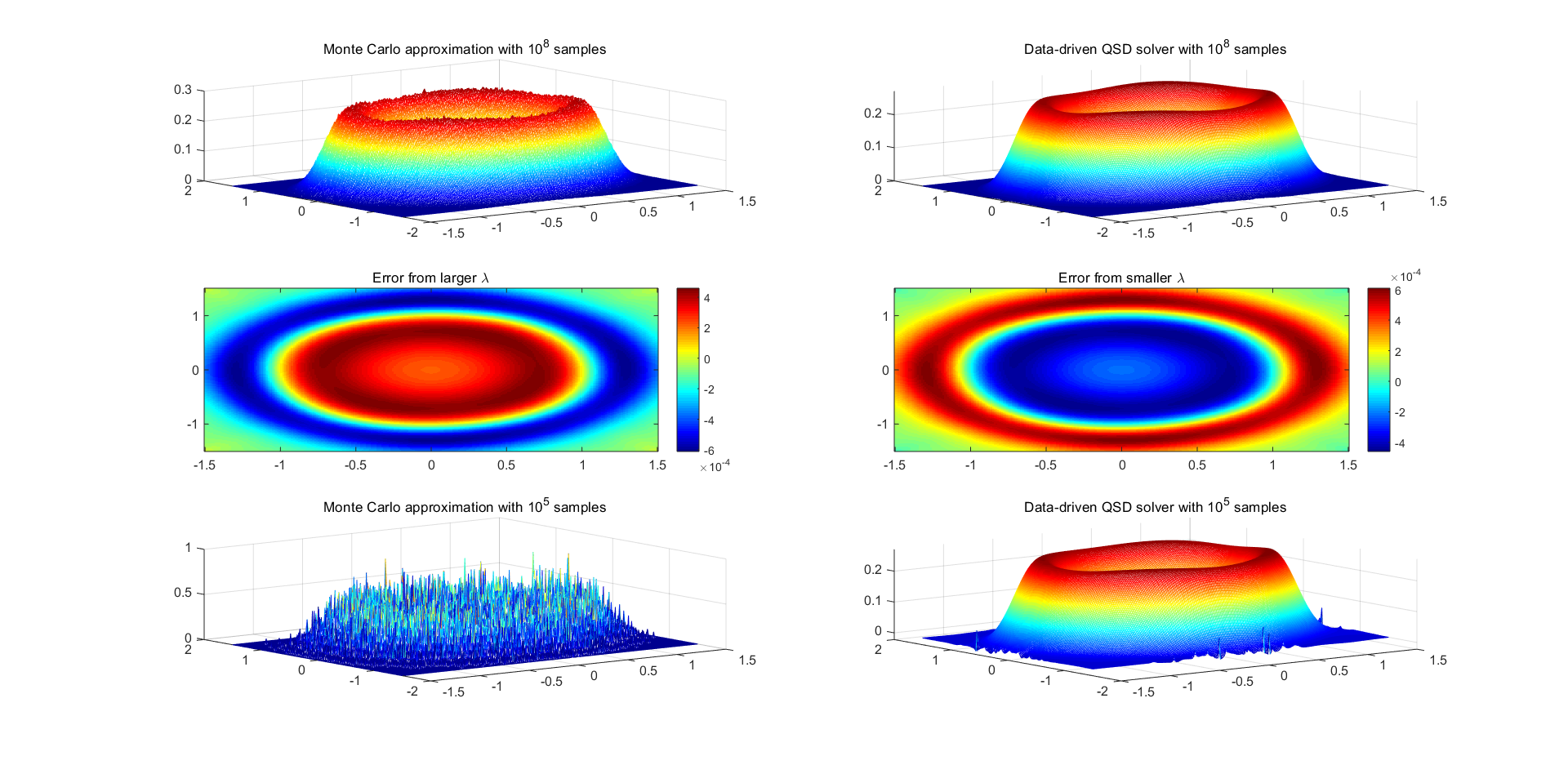}
     \caption{(Ring density) $\textbf{Upper panel:}$The approximation
       by Monte Carlo simulation(left) and the algorithm in Section
       3.2(right) with $256\times 256$ mesh points and $10^8$
       samples. $\textbf{Middle panel}$: Sensitivity effect of small
       change to killing rate $\lambda$. \textbf{Lower panel:} The
       approximation by Monte Carlo simulation with smaller
       samples(left) and the output of the data-driven QSD solver(right).}
     \label{F4}
 \end{figure} 

The simulation result shows the estimated rate of killing
$\lambda=-0.176302$. We use this example to test the sensitivity of solution
$\textbf{u}$ against small change of  the killing rate. We compare the approximations
obtained by setting the killing rate be $\lambda$,
$1.1\lambda$ and $0.9\lambda$ respectively. Heat maps
of the difference between QSDs with ``correct'' and ``wrong'' killing
rates are shown in two middle panels in Figure
\ref{F4}. It shows that difference brought by ``wrong'' rates is only
$\approx O(10^{-4})$, which can be neglected. This result coincides with the
analysis in Theorem \ref{lambda} in this paper.

Finally, we would like to emphasis that the data-driven QSD solver can
tolerate very high level of spatially uncorrelated noise in the
reference solution $\mathbf{v}$. For example, if we use the same long
trajectory with $10^{8}$ samples that generates the top left panel of
Figure \ref{F4}, but only select $10^{5}$ samples with intervals of
$10^{3}$ steps of the numerical SDE solver, the Monte Carlo data
becomes very noisy (Bottom left panel of Figure \ref{F4}). However,
longer intermittency between samples also reduces the spatial
correlation between samples. As a result, the output of the QSD solver
has very little change except at the boundary grid points, because the
optimization problem \eqref{opt} projects most of error to the
boundary of the domain. (See bottom right panel of Figure \ref{F4}.)
This result highlights the need of high quality samplers. A Monte Carlo
sampler with smaller spatial correlation between samples can
significantly reduce the number of samples needed in the data-driven
QSD solver.

\subsection{Sensitivity of QSD: 1D examples} 
In this subsection, we use 1D examples to study the sensitivity of
QSDs against changes on boundary conditions. Consider an 1D gradient
flow of the potential function $V(x)$ with an additive noise perturbation
\begin{equation}
  \label{gradient}
    X_t=-V'(X_{t})dt+\sigma dW_t.
\end{equation}
Let $(-\infty, 0]$ be the absorbing state of $X_{t}$. So if $V(0) < \infty$, 
$X_{t}$ admits a QSD, denoted by $\mu_{X}$. If we let the stochastic
process reflect at $x = 0$, the modified stochastic process, denoted
by $Y_{t}$, admits an invariant probability measure denoted by
$\mu_{Y}$. We will compare the sensitivity of $\mu_{X}$ against the
change of boundary condition for two different cases whose speed of
mixing are different, namely a single well potential function and a double
well potential function.

We choose a single well potential function $V_{1}(x) = (x-1)^{2}$ and a
double well potential function $V_{2}(x) =
x^4-4\sqrt{2}x^3+10x^2-4\sqrt{2}x+1$. The values of minima of both
$V_{1}$ and $V_{2}$ are zero. The values of $V_{1}$ and $V_{2}$ at the
absorbing state are $1$. And the height of the barrier between two
local minima of $V_{2}$ is $1$. The strength of noise is $\sigma =
0.7$ in both examples. See Figure \ref{F5} middle column for plots of these
two potential functions. In both cases, the QSD and the invariant
probability measure are computed on the domain $D = [0, 3]$. To
further distinguish these two cases, we denote the QSD of equation \eqref{gradient} with
absorbing state $x = 0$ and potential function $V_{1}(x)$
(resp. $V_{2}(x)$) by $\mu_{X}^{1}$
(resp. $\mu_{X}^{2}$) and the invariant probability measure of
equation \eqref{gradient} with reflection boundary at $(-\infty, 0]$ and potential function $V_{1}(x)$
(resp. $V_{2}(x)$) by $\mu_{Y}^{1}$
(resp. $\mu_{Y}^{2}$). Probability measures $\mu_{X}^{1}$ and
$\mu_{Y}^{1}$ (resp. $\mu_{X}^{2}$ and $\mu_{Y}^{2}$) are compared in
Figure \ref{F5} right column. We can see that the QSD and the
invariant probability measure have small difference for the single
well potential function $V_{1}$. But they look very different for the
double well potential function $V_{2}$. With the double well potential
function, there is a visible difference between probability
density functions of $\mu^{2}_{X}$ and $\mu^{2}_{Y}$. The density
function of QSD is much smaller than the invariant probability measure
around the left local minimum $x=1-\sqrt{2}$ because this local
minimum is closer to the absorbing set, which makes killing
and regeneration more frequent when $X_{t}$ is near this local
minimum. In other words, the QSD of 
equation \eqref{gradient} with respect to the double well potential is
very sensitive against the change at the boundary.

The reason of the high sensitivity is illustrated by the coupling argument. We first run Algorithm \ref{Alg:coupling}
with 8 independent long trajectories with length of $10^6$ and collect
the coupling times. The slope of exponential tail of the coupling time
gives the rate of contraction of $P^T_Y$. The $\mathbb{P}(\tau^C>t)$
versus $t$ plot is demonstrated in log-linear plot in Figure
\ref{F5} left column. The slope of exponential tail is $\gamma =
2.031414$ for the single well potential $V_{1}$, and $\gamma = 0.027521$
for the double well potential case. Then we run Algorithm
\ref{Alg:W1case1} to estimate the finite time error
$\mathrm{d}_{w}(\mu_{X}P_{X}^{T}, \mu_{X}P_{Y}^{T})$ for both
cases. Since the single well potential case has a much faster coupling
speed, we can choose $T = 0.5$. The output of Algorithm
\ref{Alg:W1case1} is $d_w(\mu_X P^T_X, \mu_X P^T_Y) \approx
0.00391083$. This gives an estimate $d_{w}(\mu_{X}, \mu_{Y}) \approx
0.0061$. The double well potential case converges much slower. We
choose $T = 20$ to make sure that the denominator $1 - e^{-\gamma T}$ is not too small. As
a result, Algorithm \ref{Alg:W1case1} gives an approximation $d_w(\mu_X P^T_X, \mu_X P^T_Y) \approx
0.06402$, which means $d_{w}(\mu_{X}, \mu_{Y}) \approx 0.1512$. This
is consistent with the right column seen in Figure \ref{F5}, the QSD
of the double well potential is much more sensitive against a change of the
boundary condition than the single well potential case.

\begin{figure}[h]
     \centering
     \includegraphics[width=\textwidth]{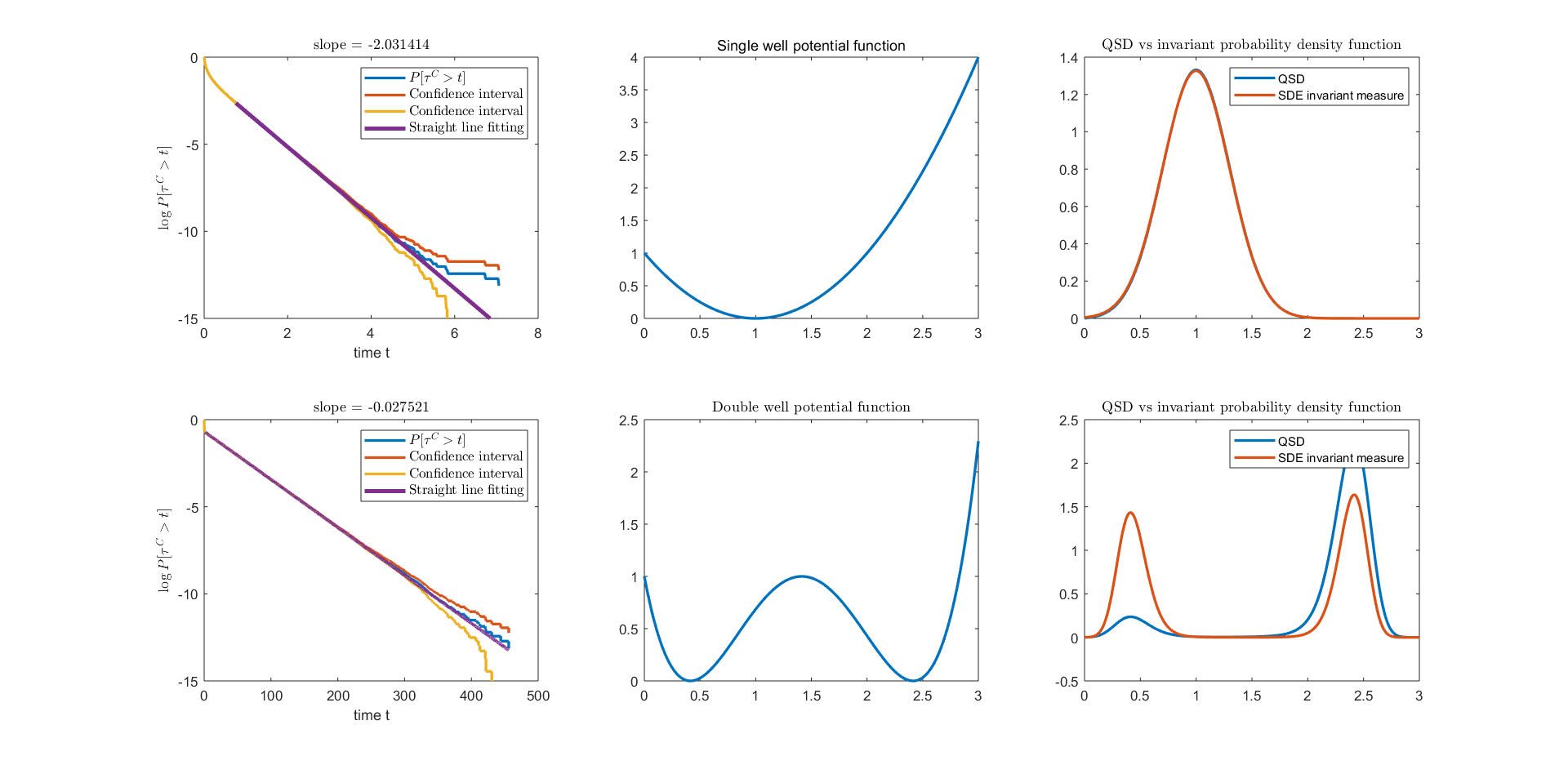}
     \caption{(Single well vs. Double well potential function) $\textbf{Left column}$: $\mathbb{P}(\tau^C>t)$ vs. $t$. $\textbf{Middle column:}$ Single well and double well potential functions. $\textbf{Right column}$: QSD vs. invariant density function. }
     \label{F5}
\end{figure}

\subsection{Lotka-Volterra Competitive Dynamics}
In this example, we focus on the effect of demographic noise on the classical Lotka-Volterra competitive system. The Lotka-Volterra competitive system with some environmental fluctuation has the form
\begin{equation}
    \begin{split}
        dY_1(t) &=Y_1(t)(l_1-a_{11}Y_1(t)-a_{12}Y_2(t))dt+\sigma'_1Y_1(t)dW_1(t),\\
        dY_2(t) &=Y_2(t)(l_2-a_{22}Y_1(t)-a_{21}Y_1(t))dt+\sigma'_2Y_2(t)dW_2(t).
    \end{split}
    \label{LV_1}
\end{equation}
Here $l_i>0$ is the per-capita growth rate of species $i$ and $a_{ij}>0$ is the per-capita competition rate between species $i$ and $j$.
More details can be found in \cite{hening2020stationary}. Model parameters are chosen to
be $l_1=2, l_2=4, a_{11}=0.8, a_{12}=1.6, a_{21}=1, a_{22}=5$. Let
$\partial \mathcal{X}$ be the union of x-axis and y-axis. For
suitable $\sigma'_{1}$ and $\sigma'_{2}$, $Y_{1}$ and $Y_{2}$ can
coexist such that the probability of $(Y_{1}, Y_{2})$ hits $\partial
\mathcal{X}$ in finite time is zero. So equation \eqref{LV_1} admits
an invariant probability measure, denoted by $\mu_{Y}$.

As a modification, we add a small demographic noise term to equation
\eqref{LV_2}. The equation becomes
\begin{equation}
\begin{split}
        dX_1(t)
        &=X_1(t)(l_1-a_{11}X_1(t)-a_{12}X_2(t))dt+\sigma_1X_1(t)dW_1(t)+
        \epsilon \sqrt{X_1(t)}dW'_1(t),\\
        dX_2(t)
        &=X_2(t)(l_2-a_{22}X_1(t)-a_{21}X_1(t))dt+\sigma_2X_2(t)dW_2(t)+\epsilon
        \sqrt{X_2(t)}dW'_2(t).
    \end{split}
    \label{LV_2}
  \end{equation}
It is easy to see that equation \eqref{LV_2} can exit to the boundary
$\partial \mathcal{X}$. It admits a QSD, denoted by $\mu_{X}$. 
  
In order to study the effect of demographic noise, we compare
$\mu_{Y}$, the numerical invariant measure of equation (\ref{LV_1}),
and $\mu_{X}$, the QSD of equation (\ref{LV_2}).  In our simulation, we fix the strength
of demographic noise as $\epsilon=0.05$ and compare $\mu_{X}$ and
$\mu_{Y}$ at two different levels of the environment noise $\sigma_{1}
= \sigma_{2}=0.75$ and $\sigma_{1} = \sigma_{2}=1.1$
respectively. The coefficient $\sigma_{1}'$ and $\sigma_{2}'$ in
equation \eqref{LV_1} satisfies $\sigma_{i}' = \sqrt{\sigma_{i}^{2}  +
\epsilon^{2}}$ for $i = 1, 2$ to match the effect of the additional
demographic noise. Compare Figure \ref{F6} and Figure \ref{F7}, one
can see that $\mu_{Y}$ has significant concentration at the boundary
when $\sigma_{1} = \sigma_{2} = 1.1$.

The result for $\sigma_{1} = \sigma_{2}= 0.75$ is shown in Figure \ref{F6}. Left bottom of
Figure \ref{F6} shows the invariant measure. The QSD is shown on
right top of Figure \ref{F6}. The total variation distance between these two
measures are shown at the bottom of Figure \ref{F6}. The difference is
very small and it just appears around boundary. This is reasonable
because with high probability, the trajectories of equation
\eqref{LV_2} moves far from the absorbing set
$\partial\mathcal{X}$ in both cases, which makes the regeneration
events happen less often. This is consistent with the result of Lemma
\ref{Lem:W1case2}. We compute the distribution of the
coupling time. The coupling time distribution and its exponential tail
are shown in Figure \ref{F6} Top Left. Then we use
Algorithm \ref{Alg:W1case2} to compute the finite time error. To
better match two trajectories given by equations \eqref{LV_1} and
\eqref{LV_2}, we separate the noise term in equation \eqref{LV_1} into
$\sigma_{i}'Y_{i}(t) \mathrm{d}W_{i}(t) = \sigma_{i}Y_{i}(t)
\mathrm{d}W^{(1)}_{i}(t) + \epsilon Y_{i}(t)
\mathrm{d}W^{(2)}_{i}(t)$ for $i = 1,2$, where $W^{(1)}_{i}(t)$ and
$W^{(2)}_{i}(t)$ use the same Wiener process trajectory as $W_{i}(t)$
and $W'_{i}(t)$ in equation \eqref{LV_2} for $i = 1, 2$. Let $T = 4$. The finite time error
caused by the demographic noise is $0.01773$. As a result, the
upper bound given in Lemma \ref{Lem:W1case2} is $0.02835$. Note that
as seen in Figure \ref{F6}, this upper bound actually significantly
overestimates the distance between the invariant probability measure
and the QSD. The empirical total variation distance is much smaller
that our theoretical upper bound. This is because the way of matching
$\sigma'_{i}Y_{i} \mathrm{d}W_{i}(t)$ and $\sigma_{i}X_{i}(t)
\mathrm{d}W_{i}(t) + \epsilon \sqrt{X_{i}(t)} \mathrm{d}W_{i}(t)$ is
very rough. A better approach of matching those noise terms will
likely lead to a more accurate estimation of the upper bound of the
error. 

The results for $\sigma_{1} = \sigma_{2} = 1.1$ are shown in Figure
\ref{F7}. The total variation distance between these two measures are shown at
the bottom of Figure \ref{F7}. It is not hard to see the difference is
significantly larger than case $\sigma=0.75$. The reason is that trajectories of
equation \eqref{LV_2} have high probability moving
along the boundary in this parameter setting. This makes the probability
of falling into the absorbing set $\partial\mathcal{X}$ much
higher. Same as above, we compute the distribution of the
coupling time and demonstrate the coupling time distribution (as well
as the exponential tail) in Figure \ref{F7} Top Left. The coupling in
this example is slower so we choose $T = 12$ to run Algorithm \ref{Alg:W1case2}. The
probability of killing before $T$ is approximately $0.11186$ and the total finite time error
caused by the demographic noise is $0.06230$. As a result, the
upper bound given in Lemma \ref{Lem:W1case2} is $0.1356$. This is
consistent with the numerical finding shown in Figure \ref{F7} Bottom
Right. 

\begin{figure}[h]
 \centering
     \includegraphics[width=\textwidth]{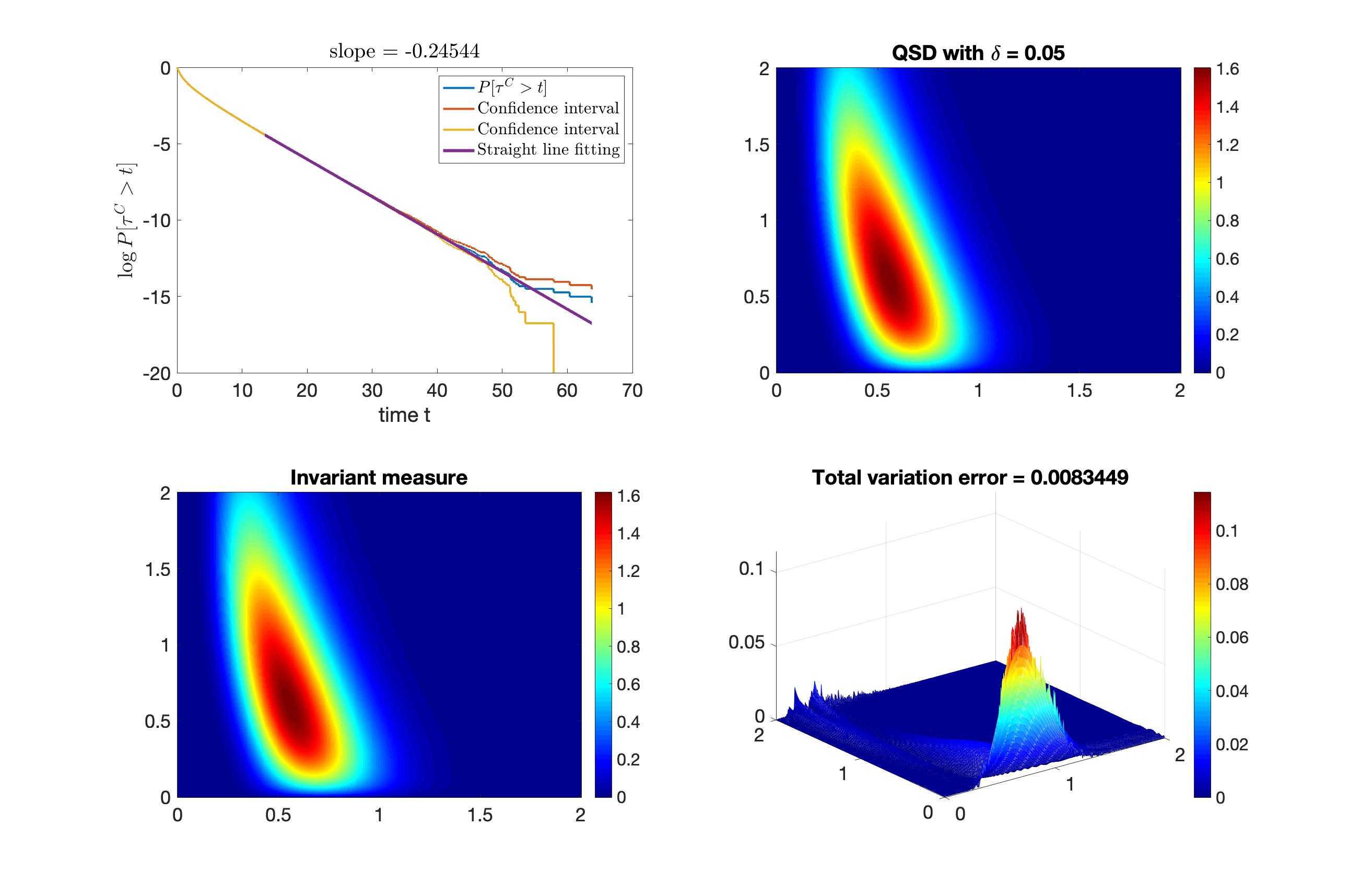}
     \caption{(Case $\sigma_1=\sigma_2=0.75$) $\textbf{Upper panel}$: ($\textbf{Left}$) $\mathbb{P}(\tau^C>t)$ vs.$t$. ($\textbf{Right}$) QSD with demographic noise coefficient $\epsilon=0.05$.\ $\textbf{Lower panel}$: ($\textbf{Left}$) Invariant density function for $\sigma=0.75$. ($\textbf{Right}$) Total variation of QSD and invariant density function.}     
    \label{F6}
 \end{figure}

\begin{figure}[h]
     \centering
     \includegraphics[width=\textwidth]{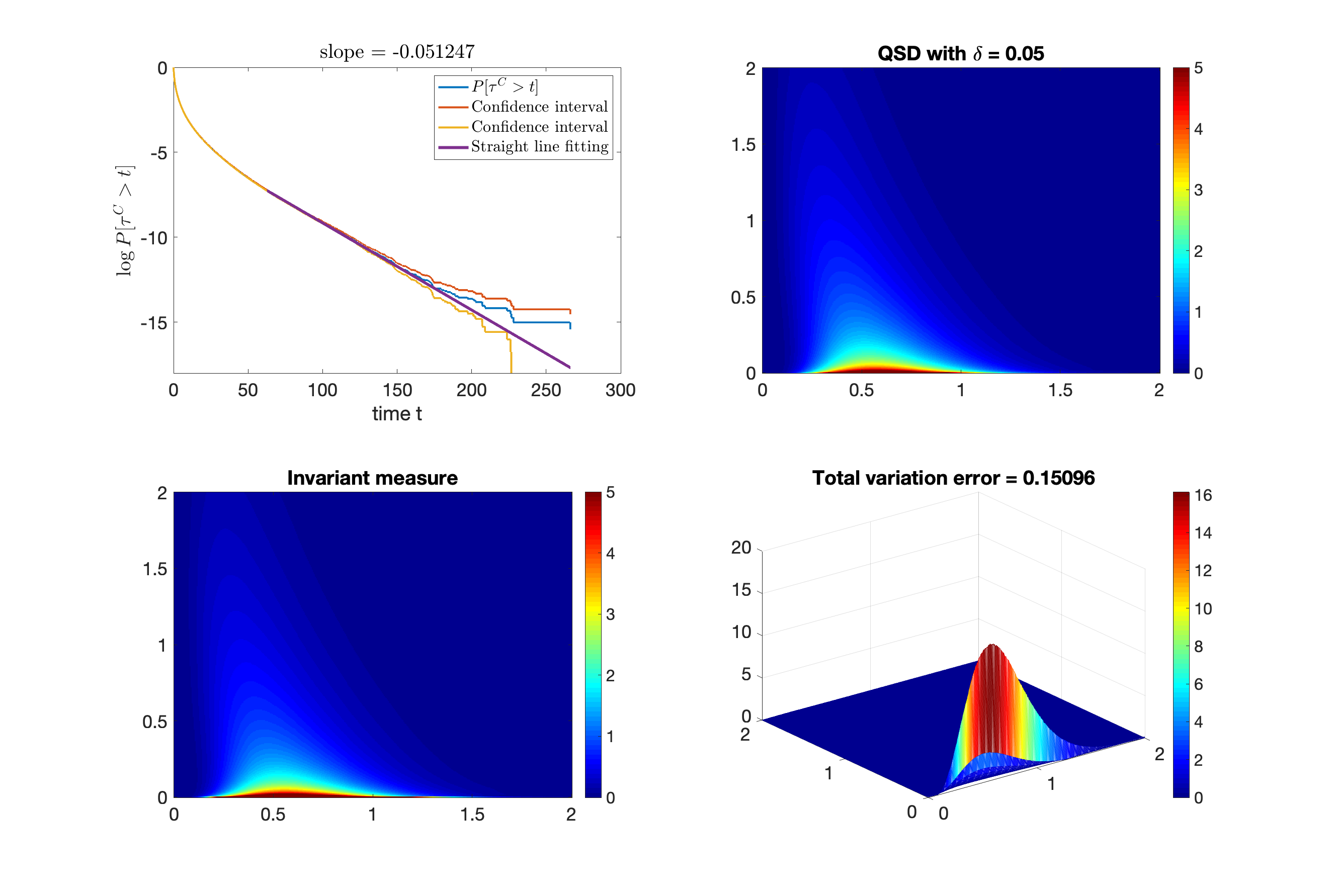}
     \caption{(Case $\sigma_1=\sigma_2=1.1$) $\textbf{Upper panel}$: ($\textbf{Left}$) $\mathbb{P}(\tau^C>t)$ vs.$t$. ($\textbf{Right}$) QSD with demographic noise coefficient $\epsilon=0.05$.\ $\textbf{Lower panel}$: ($\textbf{Left}$) Invariant density function for $\sigma=1.1$. ($\textbf{Right}$) Total variation of QSD and invariant density function.}
     \label{F7}
\end{figure}

\subsection{Chaotic attractor}
In this subsection, we consider a non-trivial 3D example that has
interactions between chaos and random perturbations, called the
Rossler oscillator. The random perturbation of the Rossler oscillator is
\begin{equation}
  \label{Rossler}
\begin{cases}
dx=(-y-z)dt+\epsilon dW^x_t\\
dy=(x+ay)dt+\epsilon dW^y_t\\
dz=(b+z(x-c))dt+\epsilon dW^z_t,
\end{cases}
\end{equation}
where $a=0.2, b=0.2, c=5.7$, and $W^x_t, W^y_t$ and $W^z_t$ are
independent Wiener processes. The strength of noise is chosen to be
$\epsilon=0.1$. This system is a representative example of chaotic ODE
systems appearing in many applications of physics, biology and
engineering. We consider equation \eqref{Rossler} restricted to the
box $D = [-15, 15]\times[-15, 15]\times[-1.5, 1.5]$. Therefore, it
admits a QSD supported by $D$. In this example, a grid with $1024\times
1024\times 128$ mesh points is constructed on $D$.

It is very difficult to use traditional PDE solver to compute a large
scale 3D problem. To analyze the QSD of this chaotic system, we apply
the blocked version of the Fokker-Planck solver
studied in \cite{dobson2019efficient}. More precisely, a big mesh is divided into many
``blocks''. Then we solve the optimization problem \eqref{opt} in
parallel. The collaged solution is then processed by the ``shifting
block'' technique to reduce the interface error, which means the
blocks are reconstructed such that the center of new blocks cover the
boundary of old blocks. Then we let the solution from the first found
serve as the reference data, and solve optimization problem
\eqref{opt} again based on new blocks. See \cite{dobson2019efficient} for the full
details of implementation. In this example, the grid is further divided into
$32\times32\times4$ blocks. We run the ``shifting block'' solver for 3
repeats to eliminate the interface error. The reference solution is generated by a
Monte Carlo simulation with $10^9$ samples. The killing rate is
$\lambda = -0.473011$. Two ``slices'' of the
solution, as seen in Figure \ref{F8}, are then projected to the
xy-plane for the sake of easier demonstration. See the caption of
Figure \ref{F8} for the coordinates of these two ``slices''. The left picture in Figure \ref{F8} shows the projection of the
solution has both dense and sparse parts that are clearly divided. An
outer "ring" with high density appears and the density decays quickly
outside this "ring." The right picture in Figure \ref{F8} demonstrates
the solution has much lower density when $z$-coordinate is larger than
1.    

\begin{figure}[ht!]
     \centering
     \includegraphics[width=\textwidth]{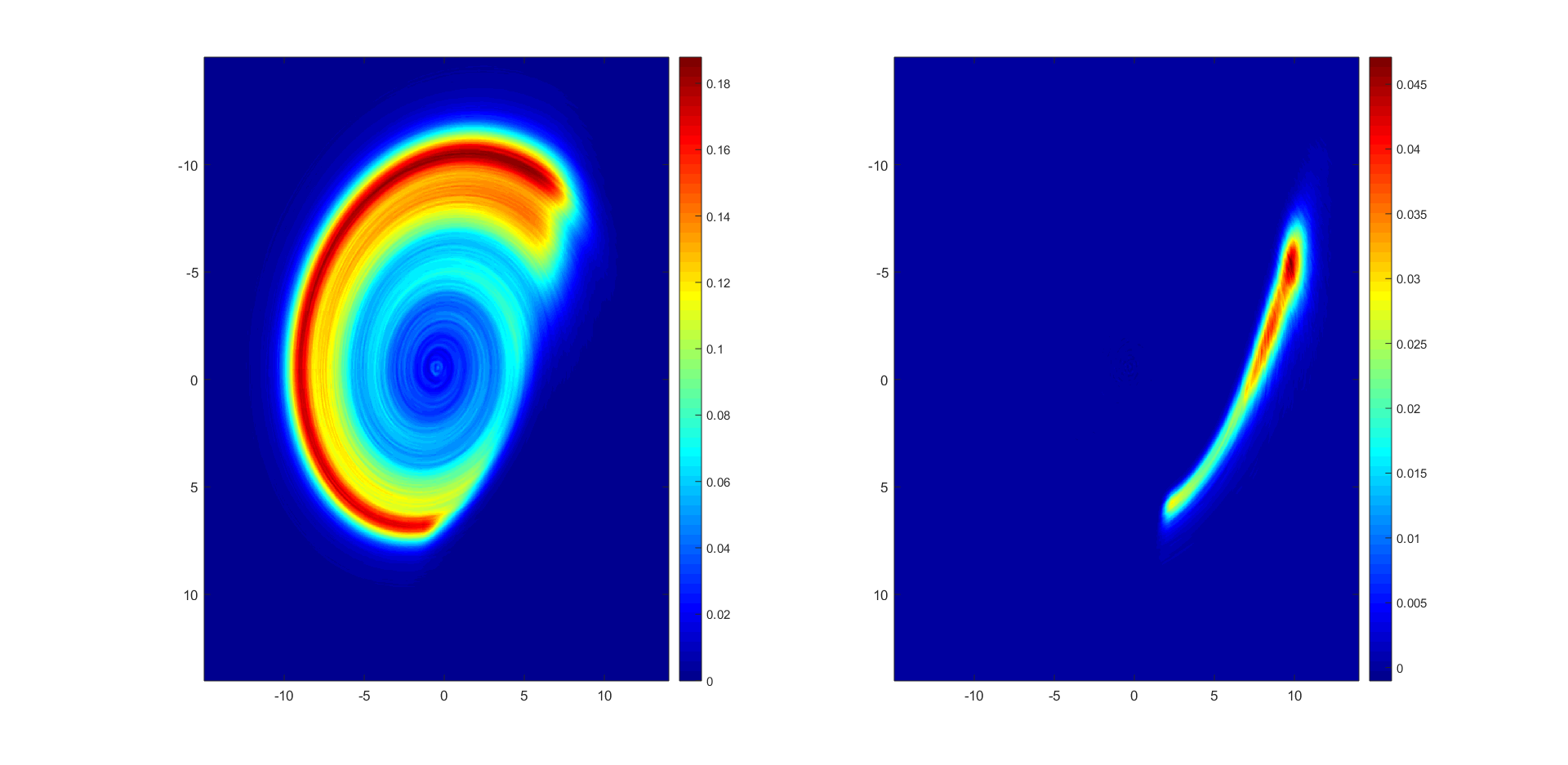}
     \caption{(Rossler) Projections of 2 "slices" of the QSD of the Rossler system to the xy-plane. z-coordinates of 2 slices are [-0.0352, 0.2578], [1.1367, 1.4297]. The solution is obtained by a balf-block shift solver on $[-15,15]\times[-15,15]\times[-1.5, 1.5]$ with $1024\times1024\times 128$ mesh points, $32\times32\times4$ blocks, and $10^9$ samples.}
     \label{F8}
\end{figure}

\section{Conclusion} In this paper we provide some data-driven
methods for the computation of quasi-stationary distributions (QSDs) and
the sensitivity analysis of QSDs. Both of them are extended from the
first author's earlier work about invariant probability measures. When
using the Fokker-Planck equation to solve the QSD, we
find that the idea of using a reference solution with low accuracy to
set up an optimization problem still works well for QSDs. And the QSD
is not very sensitively dependent on the killing rate, which is given
by the Monte Carlo simulation when producing the reference
solution. The data-driven Fokker-Planck solver studied in this paper
is still based on discretization. But we expect the mesh-free
Fokker-Planck solver proposed in \cite{dobson2019efficient} to work for
solving QSDs. In the sensitivity analysis part, the focus is on the
relation between a QSD and the invariant probability measure of a
``modified process'', because many interesting problems in
applications fall into this category. The sensitivity
analysis needs both a finite time truncation error and a contraction
rate of the Markov transition kernel. The approach of estimating the finite time truncation error
is standard. The contraction rate is estimated by using the novel numerical
coupling approach developed in \cite{li2020numerical}. The sensitivity analysis
of QSDs can be extended to other settings, such as the sensitivity
against small perturbation of parameters, or the sensitivity of a
chemical reaction process against its diffusion approximation. We will
continue to study sensitivity analysis related to QSDs in our subsequent work.

\newpage
\bibliographystyle{plain}
\bibliography{ref.bib}

\begin{thebibliography}{10}

\bibitem{agresti1998approximate}
Alan Agresti and Brent~A Coull.
\newblock Approximate is better than “exact” for interval estimation of
  binomial proportions.
\newblock {\em The American Statistician}, 52(2):119--126, 1998.

\bibitem{anderson2012continuous}
William~J Anderson.
\newblock {\em Continuous-time Markov chains: An applications-oriented
  approach}.
\newblock Springer Science \& Business Media, 2012.

\bibitem{barton1993uniform}
Russell~R Barton and Lee~W Schruben.
\newblock Uniform and bootstrap resampling of empirical distributions.
\newblock In {\em Proceedings of the 25th conference on Winter simulation},
  pages 503--508, 1993.

\bibitem{benaim2018stochastic}
Michel Benaim, Bertrand Cloez, Fabien Panloup, et~al.
\newblock Stochastic approximation of quasi-stationary distributions on compact
  spaces and applications.
\newblock {\em Annals of Applied Probability}, 28(4):2370--2416, 2018.

\bibitem{collet2012quasi}
Pierre Collet, Servet Mart{\'\i}nez, and Jaime San~Mart{\'\i}n.
\newblock {\em Quasi-stationary distributions: Markov chains, diffusions and
  dynamical systems}.
\newblock Springer Science \& Business Media, 2012.

\bibitem{darroch1965quasi}
John~N Darroch and Eugene Seneta.
\newblock On quasi-stationary distributions in absorbing discrete-time finite
  markov chains.
\newblock {\em Journal of Applied Probability}, 2(1):88--100, 1965.

\bibitem{dobson2019efficient}
Matthew Dobson, Yao Li, and Jiayu Zhai.
\newblock An efficient data-driven solver for fokker-planck equations:
  algorithm and analysis.
\newblock {\em arXiv preprint arXiv:1906.02600}, 2019.

\bibitem{dobson2021using}
Matthew Dobson, Yao Li, and Jiayu Zhai.
\newblock Using coupling methods to estimate sample quality of stochastic
  differential equations.
\newblock {\em SIAM/ASA Journal on Uncertainty Quantification}, 9(1):135--162,
  2021.

\bibitem{ferrari1992existence}
Pablo~A Ferrari, Servet Mart{\'\i}nez, and Pierre Picco.
\newblock Existence of non-trivial quasi-stationary distributions in the
  birth-death chain.
\newblock {\em Advances in applied probability}, pages 795--813, 1992.

\bibitem{hening2020stationary}
Alexandru Hening and Yao Li.
\newblock Stationary distributions of persistent ecological systems.
\newblock {\em arXiv preprint arXiv:2003.04398}, 2020.

\bibitem{huillet2007wright}
Thierry Huillet.
\newblock On wright--fisher diffusion and its relatives.
\newblock {\em Journal of Statistical Mechanics: Theory and Experiment},
  2007(11):P11006, 2007.

\bibitem{johndrow2017error}
James~E Johndrow and Jonathan~C Mattingly.
\newblock Error bounds for approximations of markov chains used in bayesian
  sampling.
\newblock {\em arXiv preprint arXiv:1711.05382}, 2017.

\bibitem{karatzas2014brownian}
Ioannis Karatzas and Steven Shreve.
\newblock {\em Brownian motion and stochastic calculus}, volume 113.
\newblock springer, 2014.

\bibitem{kloeden1992stochastic}
Peter~E Kloeden and Eckhard Platen.
\newblock Stochastic differential equations.
\newblock In {\em Numerical Solution of Stochastic Differential Equations},
  pages 103--160. Springer, 1992.

\bibitem{lai2011transient}
Ying-Cheng Lai and Tam{\'a}s T{\'e}l.
\newblock {\em Transient chaos: complex dynamics on finite time scales}, volume
  173.
\newblock Springer Science \& Business Media, 2011.

\bibitem{li2018data}
Yao Li.
\newblock A data-driven method for the steady state of randomly perturbed
  dynamics.
\newblock {\em arXiv preprint arXiv:1805.04099}, 2018.

\bibitem{li2020numerical}
Yao Li and Shirou Wang.
\newblock Numerical computations of geometric ergodicity for stochastic
  dynamics.
\newblock {\em Nonlinearity}, 33(12):6935, 2020.

\bibitem{oksendal2003stochastic}
Bernt {\O}ksendal.
\newblock Stochastic differential equations.
\newblock In {\em Stochastic differential equations}, pages 65--84. Springer,
  2003.

\bibitem{robert2013monte}
Christian Robert and George Casella.
\newblock {\em Monte Carlo statistical methods}.
\newblock Springer Science \& Business Media, 2013.

\bibitem{van1991quasi}
Erik~A Van~Doorn.
\newblock Quasi-stationary distributions and convergence to quasi-stationarity
  of birth-death processes.
\newblock {\em Advances in Applied Probability}, pages 683--700, 1991.

\bibitem{van1995geomatric}
Erik~A van Doorn and Pauline Schrijner.
\newblock Geomatric ergodicity and quasi-stationarity in discrete-time
  birth-death processes.
\newblock {\em The ANZIAM Journal}, 37(2):121--144, 1995.

\end{thebibliography}

\end{document}